\documentclass[%
a4paper,                            
11pt,                               
bibliography=totoc,                     
abstracton,                 
]
{scrartcl}

\usepackage[a4paper,left=3.0cm,right=2.5cm, top=2.5cm, bottom=3.0cm]{geometry}
\usepackage[headsepline=.4pt,footsepline=.4pt,automark,autooneside=false,]{scrlayer-scrpage}
\clearscrheadfoot
\automark[subsection]{section}
\automark*[subsubsection]{subsection}
\ihead{\scriptsize\rightmark}
\usepackage[ngerman, english]{babel}
\usepackage[T1]{fontenc}
\usepackage[utf8]{inputenc}
\usepackage{lmodern} 
\usepackage[onehalfspacing]{setspace} 

\usepackage{xspace}
\usepackage{calc}

\usepackage{amsmath}
\usepackage{amssymb}
\usepackage{amsthm,thmtools}

\usepackage{mathtools}
\usepackage{bbm}
\usepackage{bm}
\begingroup
    \makeatletter
    \@for\theoremstyle:=definition,remark,plain\do{%
        \expandafter\g@addto@macro\csname th@\theoremstyle\endcsname{%
            \addtolength\thm@preskip\parskip
            }%
        }
\endgroup

\usepackage{chngcntr}
\counterwithin*{equation}{section}

\theoremstyle{plain}
\newtheorem{theorem}{Theorem}[section]
\newtheorem{lemma}[theorem]{Lemma}
\theoremstyle{definition}

\theoremstyle{remark}
\newtheorem{remark}[theorem]{Remark}
\theoremstyle{plain}

\theoremstyle{plain}

\theoremstyle{remark}

\usepackage{tabularx}
\usepackage{booktabs}
\usepackage{multirow}
\usepackage{multicol}
\usepackage{pdflscape}

\usepackage{diagbox}

\usepackage[flushleft]{paralist}
\setdefaultenum{(i)}{(a)}{(1)}{(aa)}

\usepackage{color}

\usepackage{graphicx}
\usepackage{tikz}
\usetikzlibrary{calc,spy,backgrounds,decorations.pathreplacing,intersections}
\usepackage{todonotes}
\usepackage[final]{showkeys}

\usepackage[normalem]{ulem}
\usepackage{fancyvrb}

\usepackage{float}
\usepackage[section]{placeins}
\usepackage[format=plain]{caption}

\usepackage{hyperref}
\hypersetup{
    colorlinks,
        linkcolor={red},
    citecolor={blue},
    urlcolor={blue}
}

\mathtoolsset{showonlyrefs=true}
\SaveVerb{euler}=euler=
\SaveVerb{exact}=exact=
\SaveVerb{m1}=m1=
\SaveVerb{m2}=m2=
\SaveVerb{m3}=m3=

\def \F {\mathcal{F}}
\def \s {\sigma}
\def \b {\beta}

\newcommand{\p}{\partial}

\newcommand{\R}{\mathbb{R}}
\newcommand{\N}{\mathbb{N}}

\newcommand{\norm}[1]{\left\|{#1}\right\|}

\newcommand{\ad}{\text{ad}}
\newcommand{\diag}{\mathrm{diag}}
\newcommand{\tridiag}{\mathrm{tridiag}}
\newcommand{\vectorize}{\mathrm{vec}}

\newcommand{\Err}{\mathrm{Err}}

\newcommand{\meanErr}{\mathrm{ME}}
\newcommand{\avgMeanErr}{\mathrm{AME}}
\newcommand{\LebesgueDelta}{\Delta^{\mathrm{Leb}}_t}
\newcommand{\stepSize}{\Delta_t}

\newcommand{\dimV}{n_v}
\newcommand{\lowerV}{a_v}
\newcommand{\upperV}{b_v}
\newcommand{\gridV}{\mathbb{V}^{\dimV}_{\lowerV,\upperV}}
\newcommand{\dimX}{n_x}
\newcommand{\lowerX}{a_x}
\newcommand{\upperX}{b_x}
\newcommand{\gridX}{\mathbb{X}^{\dimX}_{\lowerX,\upperX}}

\newcommand{\comm}[2]{\left[#1,#2\right]}

\newcommand{\abs}[1]{\left|#1\right|}

\renewcommand{\P}{\mathbb{P}}

\newcommand{\CPU}{2x AMD EPYC 7301 CPU @ 2.20\,GHz\xspace}
\newcommand{\RAM}{256\,GB RAM\xspace}
\newcommand{\GPU}{NVIDIA Tesla V100 PCIe (32\,GB HBM2 RAM)\xspace}
\newcommand{\OS}{Debian GNU/Linux 10 (buster)\xspace}

\SaveVerb{matlab}=Matlab 2022a=
\newcommand{\matlab}{\protect\UseVerb{matlab}\xspace}
\SaveVerb{ga}=ga=

\SaveVerb{fmincon}=fmincon=

\newcommand{\matlabParalleltoolbox}{Parallel Computing Toolbox\xspace}
\SaveVerb{euler}=euler=
\newcommand{\euler}{\protect\UseVerb{euler}\xspace}
\SaveVerb{exact}=exact=
\newcommand{\exact}{\protect\UseVerb{exact}\xspace}
\SaveVerb{m1}=m1=
\newcommand{\mOne}{\protect\UseVerb{m1}\xspace}
\SaveVerb{m2}=m2=
\newcommand{\mTwo}{\protect\UseVerb{m2}\xspace}
\SaveVerb{m3}=m3=
\newcommand{\mThree}{\protect\UseVerb{m3}\xspace}

\begin{document}

\title{Numerical solution of kinetic SPDEs via stochastic Magnus expansion}

\author{Kevin Kamm\thanks{Dipartimento di Matematica, Universit\`a di Bologna, Bologna, Italy.
\textbf{e-mail}: kevin.kamm@unibo.it} \and Stefano Pagliarani\thanks{Dipartimento di Matematica,
Universit\`a di Bologna, Bologna, Italy. \textbf{e-mail}: stefano.pagliarani9@unibo.it} \and
Andrea Pascucci\thanks{Dipartimento di Matematica, Universit\`a di Bologna, Bologna, Italy.
\textbf{e-mail}: andrea.pascucci@unibo.it}}


\maketitle

\begin{abstract}
In this paper, we show how the Itô-stochastic Magnus expansion can be used to efficiently solve stochastic partial differential equations (SPDE) with two space variables numerically.
To this end, we will first discretize the SPDE in space only by utilizing finite difference methods and vectorize the resulting equation exploiting its sparsity.

As a benchmark, we will apply it to the case of the stochastic Langevin equation with constant coefficients, where an explicit solution is available, and compare the Magnus scheme with the Euler-Maruyama scheme. We will see that the Magnus expansion is superior in terms of both accuracy and especially computational time by using a single GPU and verify it in a variable coefficient case. Notably, we will see speed-ups of order ranging form 20 to 200 compared to the Euler-Maruyama scheme, depending on the accuracy target and the spatial resolution. 
\end{abstract}
\noindent
\textbf{Acknowledgements:}
This project has received funding from the European Union’s Horizon 2020 research and innovation programme
under the Marie Sklodowska-Curie grant agreement No 813261 and is part of the ABC-EU-XVA
project.\\\noindent
\textbf{Data Availability:}
All data generated or analysed during this study are included in this published article. In particular, the code to produce the numerical experiments is available at \url{https://github.com/kevinkamm/MagnusSPDE2D}.

\section{Introduction}\label{sec:intro}
In the recent paper \cite{KPP2021} we have introduced a stochastic version of the Magnus expansion
\cite{MR67873} (hereafter abbreviated in ME), namely an exponential representation for solutions
of matrix-valued stochastic differential equations (SDEs), in the It\^o sense, of the form
\begin{equation}\label{eq:SDE_linear_b}
   d X_t = B_t X_t dt + A_t X_t d W_t, \qquad X_0 = I_d,
\end{equation}
for non-commuting  bounded progressively measurable random matrices $A_t$ and $B_t$. These results have recently found applications in the study
of so-called signature cumulants \cite{MR4436600}, semi-linear It\^o SDEs \cite{MR4341046}, linear
SDEs on matrix Lie groups \cite{MR4386519}, stochastic modeling of motion in turbulent flows
\cite{Campana}, stability of multi-variate geometric Brownian motion \cite{Barrera} and modeling
rating transition matrices in quantitative finance \cite{KM22}.

In this note we investigate  the application of the ME to the numerical resolution of a class of hypoelliptic stochastic partial
differential equations (SPDEs) that naturally arises in physics and mathematical finance. The
deterministic prototype of such SPDEs is the classical Langevin equation
\begin{equation}\label{eq:lange}
  \frac{1}{2}\p_{vv}u_{t}+v\p_{x}u_{}-\p_{t}u_{t}=0,
\end{equation}
where the variables $t\ge 0$, $x\in\R$ and $v\in\R$ respectively stand for time, position and
velocity, and the unknown $u_{t}=u_{t}(x,v)$ stands for the density of a particle in the phase
space. Notice that \eqref{eq:lange} is a degenerate, non-uniformly parabolic PDE. Perturbations of
\eqref{eq:lange} with variable coefficients appear in linear and non-linear form in several
applications in kinetic theory (see, for instance, \cite{Lions1}, \cite{Desvillettes} and
\cite{Cercignani}); also, \eqref{eq:lange} describes path-dependent financial derivatives such as
Asian options and volatility contracts (see, for instance, \cite{Pascucci2011} and
\cite{MR2098720}).

We consider here the stochastic version of \eqref{eq:lange}, which is the kinetic SPDE
\begin{align}\label{eq:spde}
  du_{t}=\left(\frac{a_t}{2}\p_{vv}+v\p_{x}+b_t\p_{v}+c_{t}\right)u_tdt+\left(\s_t\p_v+\b_{t}\right)u_t
  dW_t,
\end{align}
where $a_{t},b_{t},c_{t},\s_{t}$ and $\b_{t}$ are non-constant (i.e., for example,
$a_{t}=a_{t}(x,v)$) and possibly random coefficients. Here $W$ denotes a Wiener process defined on
a complete probability space $(\Omega,\F,\P)$ endowed with a filtration $\left(\F_{t}\right)_{t\ge
0}$ satisfying the usual conditions. SPDE \eqref{eq:spde} naturally appears in stochastic filtering
theory: as shown in \cite{pesfiltering,Pascucci2022}, the fundamental solution of \eqref{eq:spde} is
the conditional transition density of a two-dimensional stochastic process representing the
position and the velocity of a particle under partial observation.

The numerical solution of \eqref{eq:spde} is a challenging issue, as standard techniques, such as
Euler methods, can be cumbersome and excessively time-consuming. In this paper, we want to
demonstrate how the It\^o-stochastic ME can be used as an efficient numerical tool to solve SPDEs
meanwhile exploiting modern computer architectures such as GPUs and multiple CPUs.
The theoretical
groundwork for the It\^o-stochastic ME was established in \cite{KPP2021}
 alongside some numerical experiments using a GPU to fully utilize the parallel-in-time and parallel-in-simulation features of the ME.
Here, we show how the Magnus expansion can be used iteratively to overcome the constraint of a small convergence radius in time.
Indeed, by \cite[Theorem 1]{KPP2021}, the representation $X_t = e^{Y_t}$ for the unique strong
solution to \eqref{eq:SDE_linear_b} is valid only up to a strictly positive stopping time. More
precisely, we have the following:
\begin{theorem}[\cite{KPP2021}]\label{thm:convergence}
Let $A_t$ and $B_t$ be 
bounded progressively measurable matrices in $\R^{d\times d}$ and let
$(\Omega,\mathcal{F},\P,(\mathcal{F}_t)_{t\geq 0})$ be a filtered probability space equipped with
a standard Brownian motion $W$. For $T>0$ let also $X=(X_t)_{t\in[0,T]}$ be the unique strong
solution to \eqref{eq:SDE_linear_b}. There exists a strictly positive stopping time $\tau\leq T$
such that:
\begin{compactenum}
\item $X_t$ has a real logarithm $Y_t\in \R^{d\times d}$ up to time $\tau$, i.e.
\begin{equation}\label{eq:logarith_real_2}
 X_t = e^{Y_t},\qquad 0\leq t<\tau;
\end{equation}
\item the following representation holds $\P$-almost surely:
\begin{equation}\label{eq:convergence}
Y_t = \sum_{n=1}^{\infty} Y^{(n)}_t,\qquad 0\leq t<\tau,
\end{equation}
where $Y^{(n)}$ is the $n$-th term in the stochastic ME (see also formulas
\eqref{eq:MagnusFormulas} below, in the case of constant matrices $A$ and $B$);
\item there exists a positive constant $C$, {only dependent $A$, $B$, $T$ and $d$}, such that
\begin{equation}\label{eq:estimate_tau_conv_b}
 \P (\tau \leq t) \leq C t,\qquad t\in[0,T].
\end{equation}
\end{compactenum}
\end{theorem}
The first point of Theorem \ref{thm:convergence} tells us that the ME only converges up to a
stopping time: to overcome this restriction, the numerical implementation of the ME requires to
apply it iteratively in time. Clearly, by \eqref{eq:estimate_tau_conv_b} the convergence of the ME
is problem-dependent, meaning that there is no universal best time step-size for the ME. Point
(ii) of Theorem \ref{thm:convergence} actually yields the numerical scheme by truncating the
infinite series \eqref{eq:convergence}: 
we will see that in practice it is sufficient to consider only two or three
 terms to obtain a good
degree of accuracy.  For this reason, we will not recall the entire general expression for the terms
$Y_t^{(n)}$, 
but rather provide a user guide in
Appendix \ref{sec:MagnusFormulas} on how to easily derive the following expansion formulas, up to order $3$, in the case of constant $A$ and $B$:
\begin{equation}\label{eq:MagnusFormulas}
\begin{split}
    Y_t^{1} &=
        B t + A W_t,\\
    Y_t^{2} &=
        Y_t^1 - \frac{1}{2} A^2 t + \left[B,A\right] \int_{0}^{t}{W_s ds}
                    - \frac{1}{2} \left[B,A\right] tW_t,\\
    Y_t^3 &=
        Y_t^2 + \left[\left[B,A\right],A\right] \left(\frac{1}{2}\int_{0}^{t}{W_s^2 ds}
                    -\frac{1}{2}W_t \int_{0}^{t}{W_s ds}+\frac{1}{12} t W_t^2\right)
                    \\&\quad+\left[\left[B,A\right],B\right]
                    \left(
                        \int_{0}^{t}{sW_s ds}-\frac{1}{2} t \int_{0}^{t}{W_s ds}-\frac{1}{12} t^2 W_t
                    \right),
\end{split}
\end{equation}
where $\left[A,B\right]\coloneqq AB - BA$.



\bigskip The paper is organized as follows. In Section \ref{sec:twoDim} we 
show how to discretize a parabolic-type SPDE by an SDE of the type \eqref{eq:SDE_linear_b}. In Section
\ref{sec:numerics} we show the numerical experiments in the special case of the stochastic
Langevin equation with constant coefficients where the exact solution is available in closed form.
First, the implementation and notations for the error-analysis are introduced and afterwards
several tests concerning the parameters of the ME are discussed. Then we compare the performance
of the iterated ME with that of standard Euler-Maruyama schemes. In Section \ref{sec:stochasticLangevinVar} we consider the
more general case of SPDE \eqref{eq:spde} with variable coefficients. In Appendix
\ref{sec:MagnusFormulas} we provide a heuristic derivation of the expansion formulas up to order three.

\section{SPDEs with two space-variables}\label{sec:twoDim}
In the introduction we discussed the stochastic Langevin equation as our main application for the ME.

In this section, we show how to derive a numerical scheme for a general parabolic SPDE by combining space-discretization and ME. We will perform some formal computations, which hold even for a 
fairly general class of SPDEs. Since the computations are understood in a formal manner, we do not impose any further conditions on the coefficients of the following type of SPDE
\begin{align}
    \left\{
    \begin{aligned}[c]\arraycolsep=0pt
    &
    \begin{aligned}[t]\arraycolsep=0pt
        \hspace{1em}&\hspace{-1em}
    d u_t(x,v)=
        \biggl(
    h(x,v) u_t(x,v)
    + f^x(x,v) \partial_x u_t(x,v)
    + f^v(x,v) \partial_v u_t(x,v)
    \\&\quad
    + \frac{1}{2} g^{xx}(x,v) \partial_{xx} u_t(x,v)
    + g^{xv}(x,v) \partial_{xv} u_t(x,v)
    + \frac{1}{2} g^{vv}(x,v) \partial_{vv} u_t(x,v)
    \biggr) dt
    \\&\quad
    +
    \big(
        \sigma(x,v) u_t(x,v)
        + \sigma^{x}(x,v) \partial_{x} u_t(x,v)
        + \sigma^{v}(x,v) \partial_{v} u_t(x,v)
    \big) dW_t
    \end{aligned}\\&
    u_0(x,v)=\phi(x,v).
    \end{aligned}
    \right.
\label{eq:generalSPDE}
\end{align}
For simplicity, we will only consider the case of time-independent and non-random coefficients but it is straightforward to include them in the formal computations.

In the following subsections, we will demonstrate how to derive an SDE of the form \eqref{eq:SDE_linear_b} to approximate the SPDE \eqref{eq:generalSPDE} and apply the Magnus expansion to it. In the end, we will recall the Euler-Maruyama scheme for the approximating SDE.

\paragraph*{Space discretization and Magnus expansion.}

Following \cite[Section 3.2]{KPP2021}, we will
discretize the space variables but not time. After that we will vectorize the equation to derive a matrix-valued equation. We introduce the following 
two homogeneous grids for position and velocity respectively:
$\gridX$
with $\dimX+2$ points on the subset $[\lowerX,\upperX]\subset \R$,
$\gridV$
with $\dimV+2$ points on the subset $[\lowerV,\upperV]\subset \R$:
\begin{align}
    \gridX &\coloneqq
    \left\{
        x_i^{\dimX} \in [\lowerV,\upperV] :
            x_i^{\dimV}=\lowerX+i\Delta x,\ i=0,\dots,\dimV+1
    \right\}, &&& \Delta x&\coloneqq \frac{\upperX-\lowerX}{\dimX+1},
    \label{eq:xGrid}\\
    \gridV &\coloneqq
    \left\{
        v_j^{\dimV} \in [\lowerV,\upperV] :
            v_j^{\dimV}=\lowerV+j\Delta v,\ j=0,\dots,\dimV+1
    \right\}, &&& \Delta v&\coloneqq \frac{\upperV-\lowerV}{\dimV+1}.
    \label{eq:vGrid}
\end{align}
Let us 
denote by $\vectorize$ the isomorphism of transforming a matrix to a larger column-vector by stacking each column in the matrix below each other, i.e.
\begin{align*}
    \vectorize:\R^{\dimX\times\dimV} &\rightarrow \R^{\dimX\cdot \dimV\times 1},\\
    A=[a_{ij}] &\mapsto \vectorize(A)\coloneqq
    [a_{1,1},\dots,a_{\dimX,1},a_{1,2},\dots,a_{\dimX,2},\dots,a_{1,\dimV},\dots,a_{\dimX,\dimV}]^T.
\end{align*}

Moreover, let us define
\begin{align*}
    u_t^{\dimX,\dimV}\coloneqq
    \left(
        u_t(x_i,v_j)
    \right)_{\substack{i=1,\dots,\dimX\\j=1,\dots,\dimV}}, &&
    \phi^{\dimX,\dimV}\coloneqq
        \left(
            \phi(x_i,v_j)
        \right)_{\substack{i=1,\dots,\dimX\\j=1,\dots,\dimV}} , &&
    \Phi^{\dimX\dimV}\coloneqq
        \vectorize\left(
            \phi^{\dimX,\dimV}
        \right).
\end{align*}
We consider the $\dimX\dimV$-dimensional SDE
\begin{align}
    d  U^{\dimX\dimV}_{t} 
    =
    B  U^{\dimX\dimV}_{t} dt +
    A  U^{\dimX\dimV}_{t} dW_t, \quad
     U^{\dimX\dimV}_{0}=\Phi^{\dimX\dimV},
    \label{eq:SPDEvec}
\end{align}
where $A$ and $B$ are $(\dimX\dimV \times \dimX\dimV)$-matrices, which will be defined via space-finite differences in the next paragraph in a way that $U^{\dimX\dimV}_{t}\approx\vectorize\left( u_t^{\dimX,\dimV}   \right)$ with respect to a suitable norm.
This equation can be solved by first computing its fundamental solution, then by multiplying it with the initial datum, i.e.
\begin{align*}
    dX_t &= B X_t dt + A X_t dW_t, \quad X_0 = I \in \R^{\dimX\dimV\times\dimX\dimV},\\
     U^{\dimX\dimV}_{t} &= X_t \Phi^{\dimX\dimV}. 
\end{align*}
The fundamental solution now can be approximated with the Magnus expansion, i.e. $X_t\approx \exp\left(Y_t\right)$.
Since $A$ and $B$ 
will be very large in this case, we will utilize the sparsity of $A$ and $B$, as well as using a special algorithm specifically designed to compute the matrix-exponential times a vector denoted by \verb+expmvtay2+, which does not need to compute the whole matrix-exponential first. This is crucial for the implementation and explained in further detail in \cite{GPUExpmv2022}. The approximation formulas for $X_t$, $Y_t$ and their derivation are deferred to Appendix \ref{sec:MagnusFormulas}.

\paragraph*{Computation of $A$ and $B$.}

The idea is to discretize the first and second-order derivatives we will use central differences with
zero-boundary conditions.
Therefore, let us introduce the following matrices corresponding to the finite differences
\begin{align*}
    D^{x} &\coloneqq \frac{1}{2\Delta x} \tridiag^{\dimX,\dimX}\left(-1,0,1\right),&
    D^{v} &\coloneqq \frac{1}{2\Delta v} \tridiag^{\dimV,\dimV}\left(-1,0,1\right),\\
    D^{xx} &\coloneqq \frac{1}{\left(\Delta x\right)^2} \tridiag^{\dimX,\dimX}\left(1,-2,1\right),&
    D^{vv} &\coloneqq \frac{1}{\left(\Delta v\right)^2} \tridiag^{\dimV,\dimV}\left(1,-2,1\right).\\
\end{align*}
Additionally, we will introduce the following matrices corresponding to the coefficient functions on the discretized spatial grid
\begin{align*}
    Z^{w} &\coloneqq \left(z^w(x_i,v_j)\right)_{\substack{i=1,\dots,\dimX \\j=1,\dots,\dimV}}, &
    \Sigma^{w} &\coloneqq \left(\sigma^{w}(x_i,v_j)\right)_{\substack{i=1,\dots,\dimX \\j=1,\dots,\dimV}}
\end{align*}
for $Z=F,G,H$, $z=f,g,h$, respectively, and $w\in \left\{x,v,xx,xv,vv\right\}$.



Let us start with discretizing the first-order derivative with respect to $x$. First, we replace the partial derivative by the first-order central differences and assume zero-boundary conditions, leading to
\begin{align*}
    f^x(x_i,v_j)\partial_x u_t(x_i,v_j)&\approx
    f^x(x_i,v_j) \frac{u_t(x_{i+1},v_j)-u_t(x_{i-1},v_j)}{2\Delta x}
\end{align*}
for all $i=1,\dots,\dimX$ and $j=1,\dots,\dimV$.
As aforementioned, we need to extract the correct coefficient matrix for the vectorized equation \eqref{eq:SPDEvec}.

In our notations, a derivative in $x$ is a multiplication of the corresponding finite-difference matrix from the left to $u_t^{\dimX,\dimV}$, i.e.
\begin{align*}
    \left(\frac{u_t(x_{i+1},v_j)-u_t(x_{i-1},v_j)}{2\Delta x}\right)_{\substack{i=1,\dots,\dimX \\j=1,\dots,\dimV}}=
    D^x u_t^{\dimX,\dimV}.
\end{align*}
Using the Kronecker product, it is well-known for compatible matrices $D_1 U D_2 = C$
that
\begin{align*}
    \vectorize\left(C\right)=
    \vectorize\left(D_1 U D_2\right)=
    \left(D_2^T \otimes D_1\right)\vectorize\left(U\right).
\end{align*}
In our case, this leads to
\begin{align*}
    \vectorize\left(D^x u_t^{\dimX,\dimV}\right)=
    \left(I_{\dimV}\otimes D^x\right) U_t^{\dimX\dimV}.
\end{align*}

Now, we need to deal with the coefficients as well. Denoting by $\odot$ the
Hadamard, or elementwise, product, it is easy to see that
\begin{align*}
    \vectorize\left(F^x \odot \left(D^x  u_t^{\dimX,\dimV}\right)\right)=
    \vectorize\left(F^x \right)\odot
    \vectorize\left(\left(D^x   u_t^{\dimX,\dimV}\right)\right)=
    \diag\left(\vectorize\left(F^x \right)\right) 
    \vectorize\left(\left(D^x   u_t^{\dimX,\dimV}\right)\right).
\end{align*}
Using these two observations together yields
\begin{align*}
    \left(f^x(x_i,v_j)\partial_x u_t(x_i,v_j)\right)_{\substack{i=1,\dots,\dimX \\j=1,\dots,\dimV}}\approx
    \diag\left(\vectorize\left(F^x\right)\right) \left(I_{\dimV}\otimes D^x\right)   U_t^{\dimX\dimV}.
\end{align*}
This reasoning holds true for all other derivatives as well, i.e. an operation in $x$ is a matrix multiplication from the left and in $v$ it is the matrix multiplication from the right with the transposed matrix.

Conclusively, we have
    \begin{align*}
    B&\coloneqq
    \diag\left(\vectorize\left(H\right)\right)
    +
    \diag\left(\vectorize\left(F^x\right)\right) \left(I_{\dimV}\otimes D^x\right)
    +
    \diag\left(\vectorize\left(F^v\right)\right) \left(D^v\otimes I_{\dimX}\right)
    \\&\quad+
    \frac{1}{2}\diag\left(\vectorize\left(G^{xx}\right)\right) \left(I_{\dimV}\otimes D^{xx}\right)
    +
    \diag\left(\vectorize\left(G^{xv}\right)\right) \left(D^v\otimes D^x\right)
    \\&\quad+
    \frac{1}{2}\diag\left(\vectorize\left(G^{vv}\right)\right) \left(D^{vv}\otimes I_{\dimX}\right),\\
    A&\coloneqq
    \diag\left(\vectorize\left(\Sigma\right)\right)
    +
    \diag\left(\vectorize\left(\Sigma^{x}\right)\right) \left(I_{\dimV}\otimes D^x\right)
    +
    \diag\left(\vectorize\left(\Sigma^{v}\right)\right) \left(D^v\otimes I_{\dimX}\right).
\end{align*}
\begin{figure}%
\includegraphics[width=.32\columnwidth]{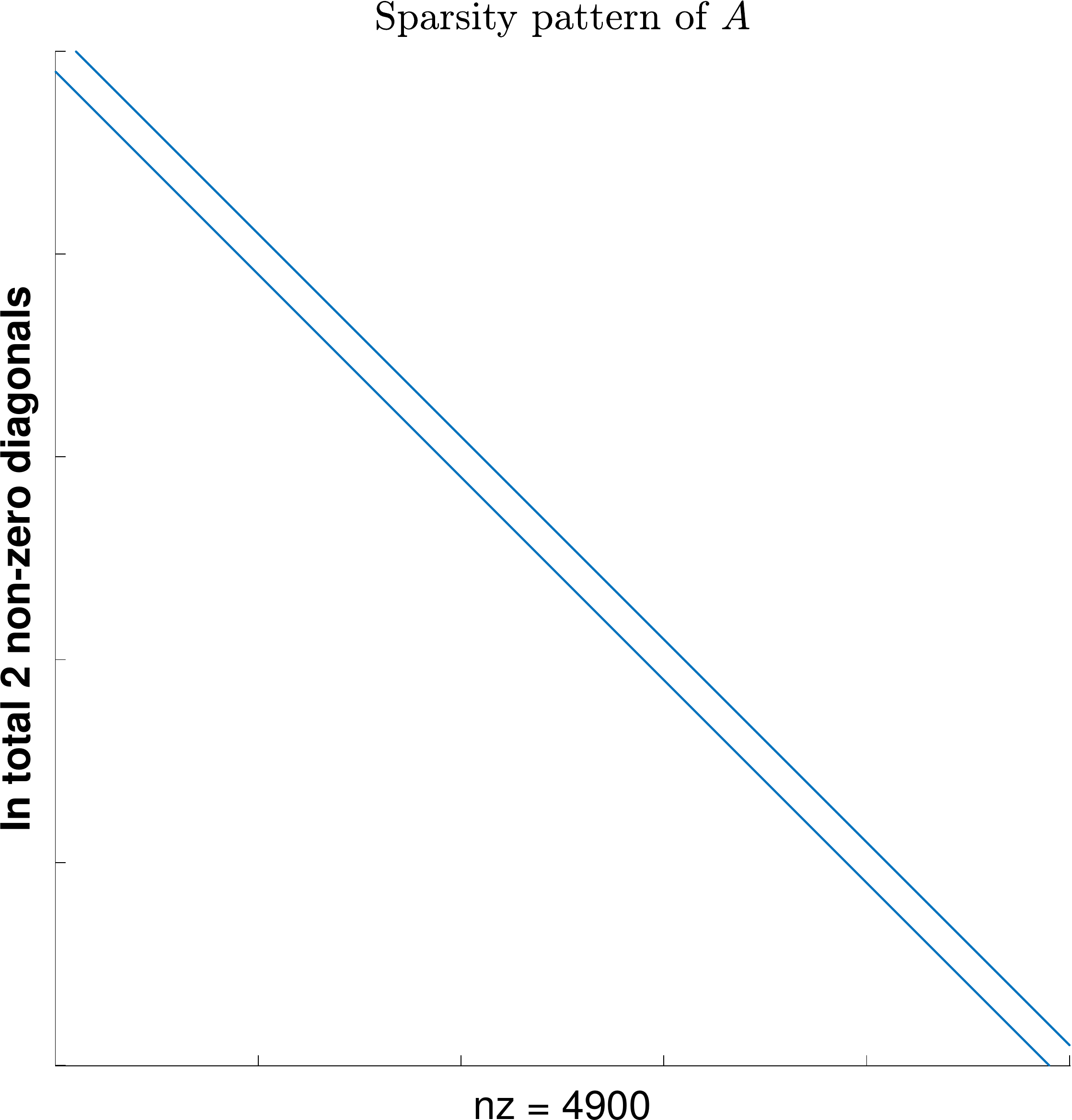}%
\includegraphics[width=.32\columnwidth]{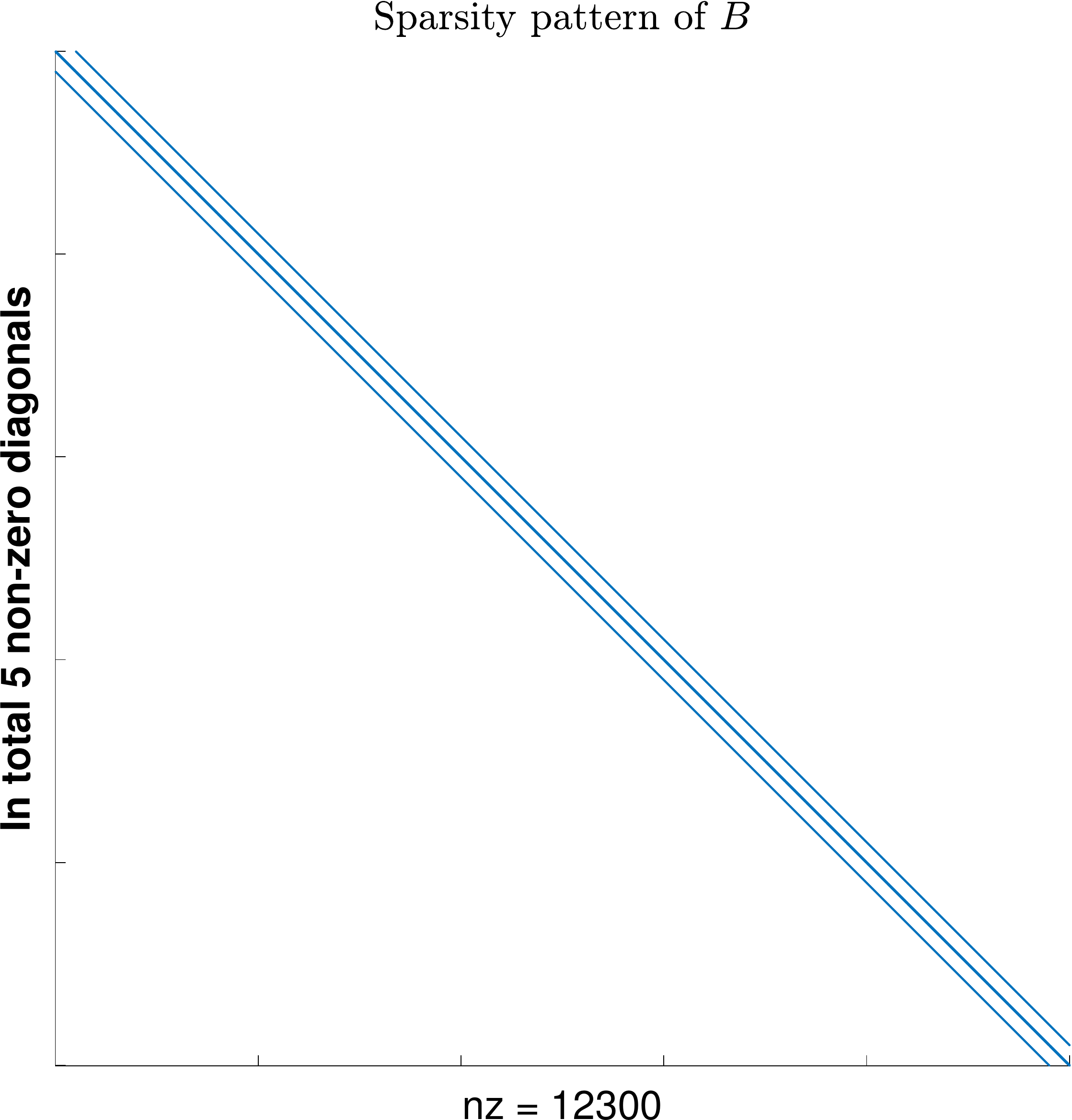}%
\includegraphics[width=.32\columnwidth]{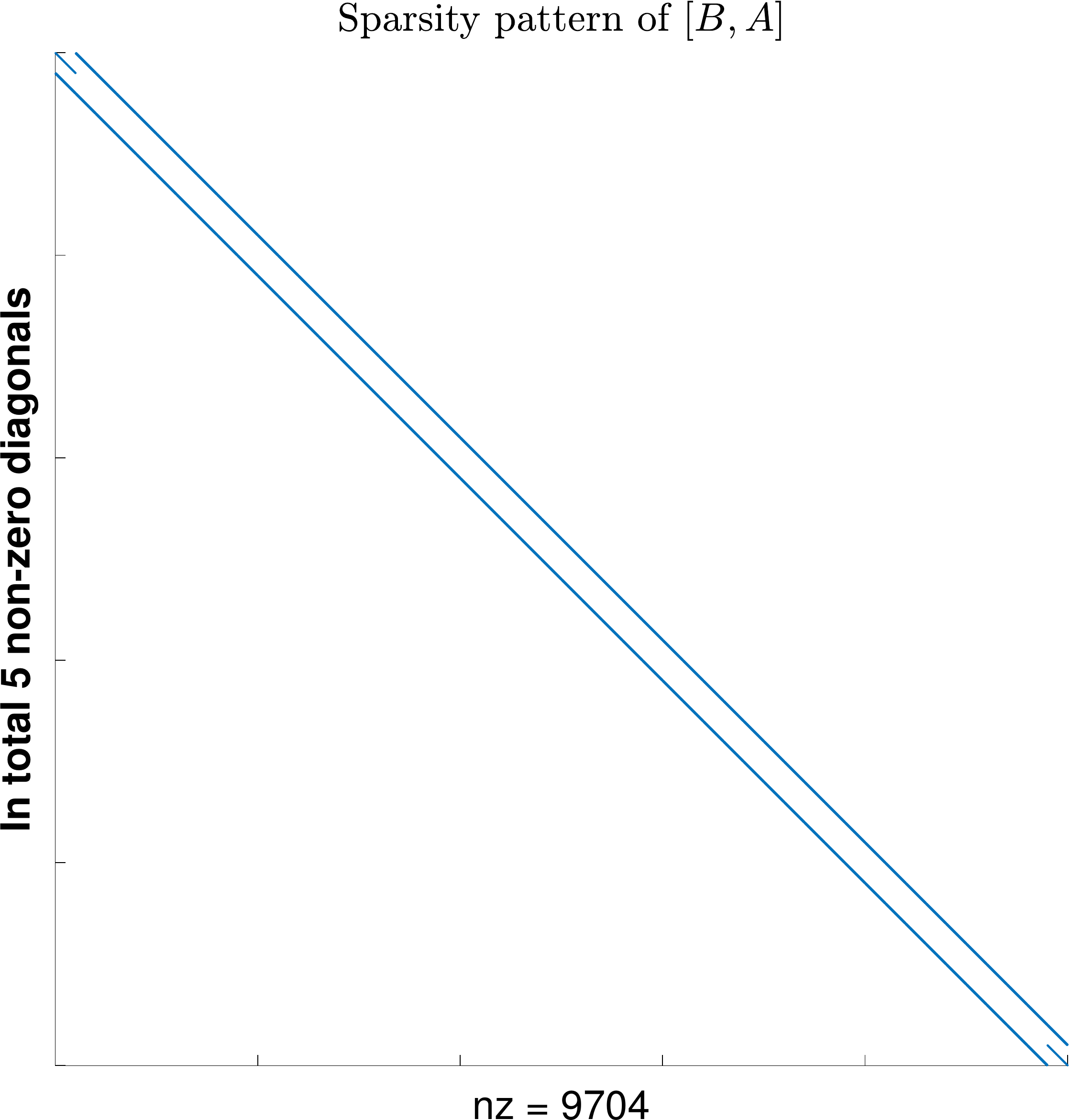}%

\includegraphics[width=.49\columnwidth]{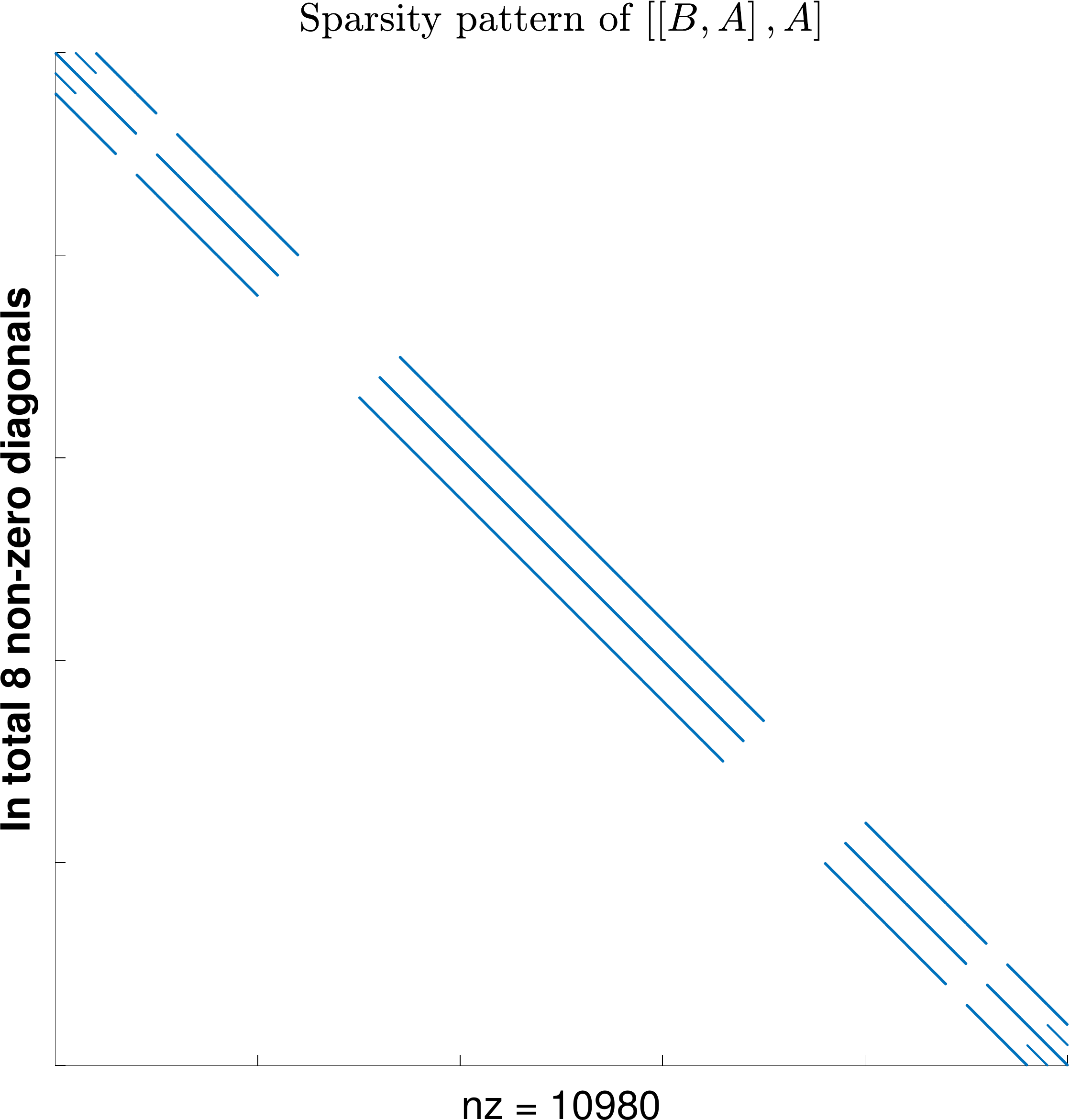}%
\includegraphics[width=.49\columnwidth]{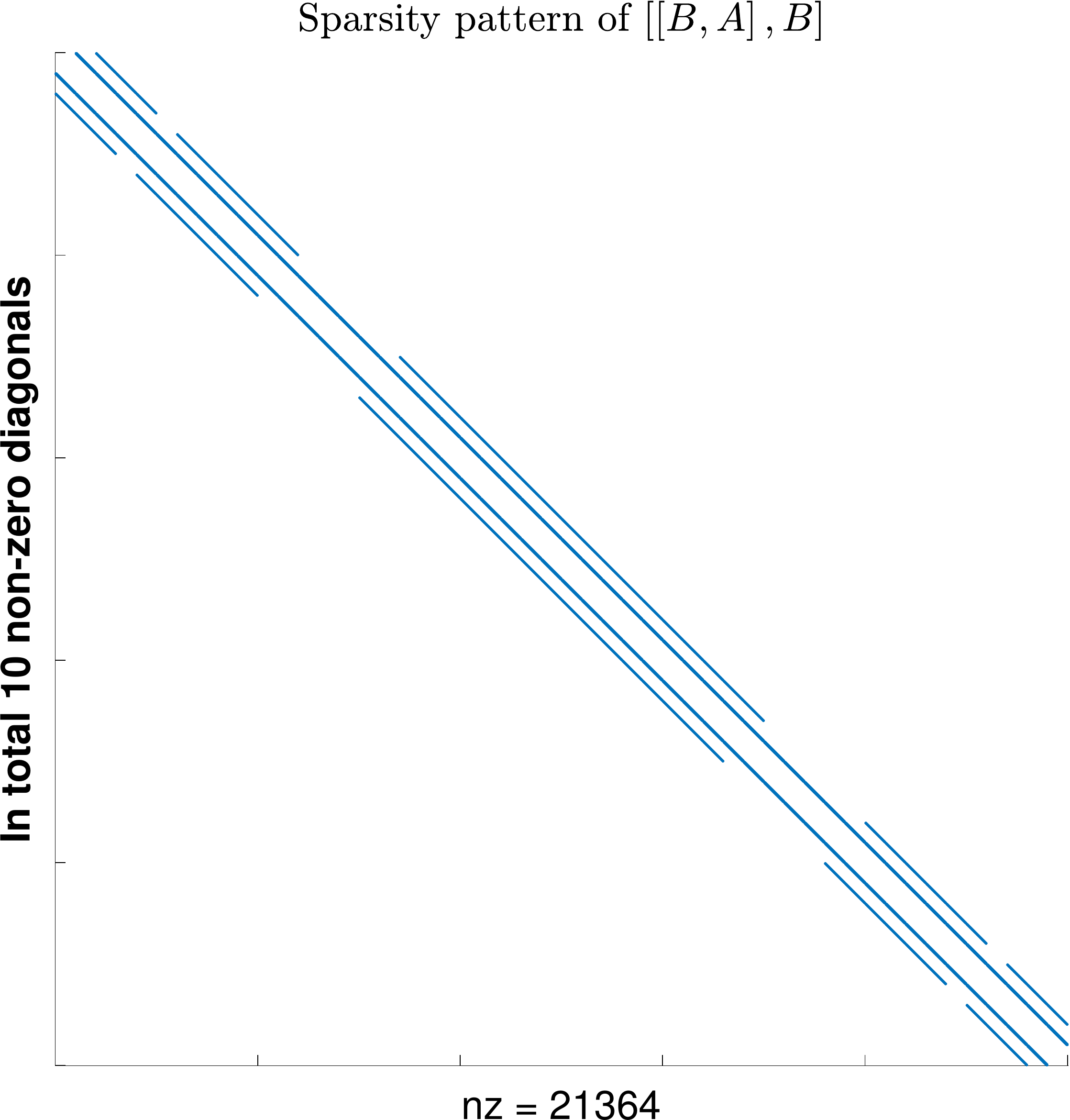}%
\caption{Sparsity patterns of $A$, $B$ and commutators of the SPDE from Section \ref{sec:stochasticLangevin} with coefficients \eqref{eq:coeffsLangevinConst} and $d=50$. nz stands for the number of non-zero entries.}%
\label{fig:sparsity}%
\end{figure}
\begin{remark}\label{rem:sparsity}%
    We can see that the sparsity of $B$ and $A$ is mostly determined by the finite-differences matrices due to the elementwise product of the coefficient functions.
    The coefficient functions can further reduce or increase the sparsity by zero
    entries or cancellations. Additionally, this implies that the number of non-zero diagonals is independent of the size of $B$ and $A$
    respectively even though the density of non-zero elements will decline for increasing dimensions.
    %

    An illustration of
    the sparsity pattern can be seen in Figure \ref{fig:sparsity}:
    In the upper left corner is the pattern for $A$, followed by $B$, $\left[B,A\right]$, $\left[\left[B,A\right],A\right]$ and $\left[\left[B,A\right],B\right]$. The blue lines represent non-zero entries. The pictures are ordered by the number of non-zero diagonals, i.e. a diagonal with at least one non-zero entry, which are 2, 5, 5, 8 and 10 respectively.
    We can see that the sparsity decreases with the order of the commutator. 
    In Section \ref{sec:stochasticLangevinVar} we will see that one consequence of the reduced sparsity is a decrease in computational efficiency of the Magnus expansion.
\end{remark}
\paragraph*{Iterated Magnus expansion}
\begin{figure}%
\centering
\includegraphics[width=.5\columnwidth]{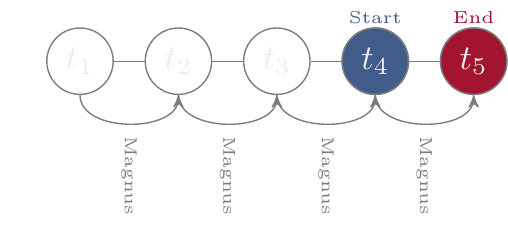}%
\caption{Schematic of the iterative Magnus scheme}%
\label{fig:iteratedMagnus}%
\end{figure}

An advantage of the Magnus expansion over pure time-iterative methods is that, as long as it is convergent, it does not need a previous time-iteration to compute the next time step.
However, as we discussed in the introduction, the convergence time-interval can be small for the Magnus expansion. Therefore, if $T>0$ is larger than the convergence radius, it is necessary to split the initial time interval $[0,T]$
into smaller sub-intervals $(0,t_1], (t_1,t_2], \dots, (t_{n},T]$. As illustrated in
Figure \ref{fig:iteratedMagnus}, we will evaluate the Magnus expansion consecutively on each of them, i.e. use the terminal evaluation as the initial point of the next sub-interval.

We will call this method \emph{iterated Magnus expansion}.
On each sub-interval, the Magnus expansion still has the usual parallel-in-time features. Furthermore, due to its relatively large convergence region in time, we can use fewer iterations compared to other iterative methods, e.g. for the Euler-Maruyama scheme.

\paragraph*{Euler-Maruyama}
In this case, we {do not need to vectorize the equation but we} 
have to discretize the time-derivative as well. With the same notation and reasoning from above we obtain
\begin{align}
    u_{t_{k+1}}^{\dimX,\dimV}&\approx
    u_{t_{k}}^{\dimX,\dimV}+
    \Biggl(
        H \odot u_{t_{k}}^{\dimX,\dimV}+
        F^x \odot \left(D^x u_{t_{k}}^{\dimX,\dimV}\right)+
        F^v \odot \left(u_{t_{k}}^{\dimX,\dimV}\left(D^v\right)^T\right)\\&\qquad+
        \frac{1}{2}G^{xx} \odot \left(D^{xx} u_{t_{k}}^{\dimX,\dimV}\right)+
        G^{xv} \odot \left(D^x u_{t_{k}}^{\dimX,\dimV} \left(D^v\right)^T\right)+
        \frac{1}{2}G^{vv} \odot \left(u_{t_{k}}^{\dimX,\dimV}\left(D^{vv}\right)^T\right)
    \Biggr)\Delta t\\&\quad+
    \Biggl(
        \Sigma \odot u_{t_{k}}^{\dimX,\dimV}+
        \Sigma^{x} \odot \left(D^x u_{t_{k}}^{\dimX,\dimV} \right)+
        \Sigma^{v} \odot \left(u_{t_{k}}^{\dimX,\dimV} \left(D^v\right)^T\right)
    \Biggr)
    \Delta W_{t_k},
    \label{eq:EulerMaruyama}
\end{align}
where $\Delta t: = t_{k+1}-t_{k} >0$ for any $k$ and $\Delta W_{t_k}:=W_{t_{k+1}}-W_{t_k}$.
In the case of separable coefficients the Hadamard product can be replaced by matrix products with diagonal matrices from the left and right.
\section{Numerical experiments}\label{sec:numerics}

We present here some numerical tests that demonstrate how the Magnus expansion can be efficiently applied  to approximate the solutions of stochastic partial differential equations (SPDEs) with two spatial variables.

In this paper, we will focus on the stochastic Langevin equation. First, we consider the case of constant coefficients. This enables us to use an exact solution as a benchmark. We will perform experiments to test the accuracy of the approximate solutions, elaborate on the computational times and discuss effects of the spatial boundaries.
Afterwards, we will consider the stochastic Langevin equation with non-constant coefficients.

To this end, we will first elaborate on the implementation of the code and introduce some notation and norms for the error-analysis.

\paragraph*{Implementation}
As mentioned beforehand, we implement the algorithm iteratively as illustrated in Figure \ref{fig:iteratedMagnus}. Now, we explain how to implement
the algorithm for one time interval in \matlab. For this, let us have a close look at \eqref{eq:MagnusFormulas}. We can see that, in this particular case, the computations of the Lebesgue-integrals are decoupled from the coefficients $A$ and $B$. This makes it possible to evaluate the Lebesgue-integrals for all trajectories in parallel by using vectorization. In particular, we chose simple Riemann-sums to compute the integrals.

Now, we need to compute the commutators of the scheme. We found that for larger problems, e.g. $d\geq 100$ it is faster to copy $A$ and $B$ to the device and compute the commutators there.
As it turns out, \matlab as of now does not support three-dimensional sparse matrices. Therefore, we decided to loop over all desired time evaluations (for our experiments only the terminal time) within the current sub-interval and over all simulations. Since, the Magnus expansion is parallel-in-time and parallel-in-simulations in each iteration, we use a threading environment to keep a single GPU busy with computing the expansion formulas and afterwards the matrix-vector exponential. This is also more efficient in terms of memory usage than using vectorization. As a side-note, it is straightforward to use multiple GPUs.

For $d\leq 100$ we use CPUs only because it is faster than copying to device in each iteration. For larger increasing spatial resolution the advantage of a single GPU increases more and more compared to CPUs only.

The choice of algorithm for evaluating the matrix-vector exponentials turned out to be crucial. We tested \verb+expmv+ by \cite{Higham2011} and \verb+expmvtay2+ by \cite{GPUExpmv2022}, as well as a Krylov-subspace implementation called \verb+expv+.
For our purposes, \verb+expmvtay2+ was more accurate and significantly faster than the other methods. Their algorithm makes it possible to utilize the sparsity of the problem and is GPU-applicable without first computing the entire matrix exponential.

We tried to improve the performance of the Euler scheme as much as possible and tested different ways to compute the matrix multiplications within each iteration of the scheme and only save the values at the terminal points for a further speed-up. For the matrix multiplications we tested full and sparse matrix multiplications on GPU and CPUs. For computations using sparse matrices, one has to loop over the simulations, since there is no analogue to \verb+pagemtimes+ in the full case.
Copying from host to device with our test cases was more expensive than a pure CPU implementation and the low-level multi-threading of \verb+pagemtimes+ surpassed the benefit of sparsity as well. Therefore, we use full matrix multiplication using all available CPU cores. On the machine we use, we could choose between 1 and 64 cores. There was no benefit after using 12 cores, which we use for all of our experiments.

We fix in our tests the number of simulations to $M=100$. If we increase the number of simulations, both methods will benefit from more available computer hardware as expected. However, even for $M=100$, the Magnus method would benefit immediately from another GPU, since the computations of matrix-vector exponentials are costly, which is the bottleneck of our implementation.

\paragraph*{Error and notations}\label{sec:Error}
For the numerical error analysis we will make use of the following notations.
We denote by $U^{\text{ref}}$ and
by $U^{\text{app}}$ a benchmark and an approximate solution, respectively.

Henceforth, we will choose $\gridV=\gridX$ symmetric, centered around zero and set 
$d\coloneqq\dimX=\dimV$ to make our analysis a bit easier.
Moreover, we choose the cut-off region of $\R^2$ as $\lowerX=\lowerV=-4$ and $\upperX=\upperV=4$ in all experiments.

Furthermore, we want to study the impact of the zero-boundary condition and therefore will use only a central part with varying size of the whole solution matrix at a given time, which is illustrated in Figure \ref{fig:kappa}. To vary the size we introduce a new parameter $\kappa=0,1,\dots$ indicating the $2^{-\kappa}$-th central part of the solution matrix $U_t^d$, which we will consider for our error analysis and denote the corresponding truncated matrix by
$U_t^{\text{ref},d,\kappa}, U_t^{\text{app},d,\kappa}\in \R^{\left\lfloor \frac{d}{2^\kappa}\right\rfloor \times \left\lfloor \frac{d}{2^\kappa}\right\rfloor}$. Also we set
\begin{align*}
    I^\kappa\coloneqq
        \left\{
            \left\lfloor \frac{d}{2}-\frac{d}{2^{\kappa+1}}\right\rfloor,\dots,
            \left\lfloor \frac{d}{2}+\frac{d}{2^{\kappa+1}}\right\rfloor
        \right\}
\end{align*}
to collect the corresponding indices.

\tikzset{kappaNode1/.style={draw=none, rectangle, align=center,text width=2.5cm}}
\tikzset{kappaLine1/.style={thick}}
\tikzset{kappaHorLine1/.style={draw=red,kappaLine1}}
\tikzset{kappaVertLine1/.style={draw=blue,kappaLine1}}
\tikzset{
    ncbar angle/.initial=90,
    ncbar/.style={
        to path=(\tikztostart)
        -- ($(\tikztostart)!#1!\pgfkeysvalueof{/tikz/ncbar angle}:(\tikztotarget)$)
        -- ($(\tikztotarget)!($(\tikztostart)!#1!\pgfkeysvalueof{/tikz/ncbar angle}:(\tikztotarget)$)!\pgfkeysvalueof{/tikz/ncbar angle}:(\tikztostart)$)
        -- (\tikztotarget)
    },
    ncbar/.default=0.5cm,
}

\tikzset{square left brace/.style={ncbar=0.5cm}}
\tikzset{square right brace/.style={ncbar=-0.5cm}}

\tikzset{round left paren/.style={ncbar=0.5cm,out=120,in=-120}}
\tikzset{round right paren/.style={ncbar=0.5cm,out=60,in=-60}}
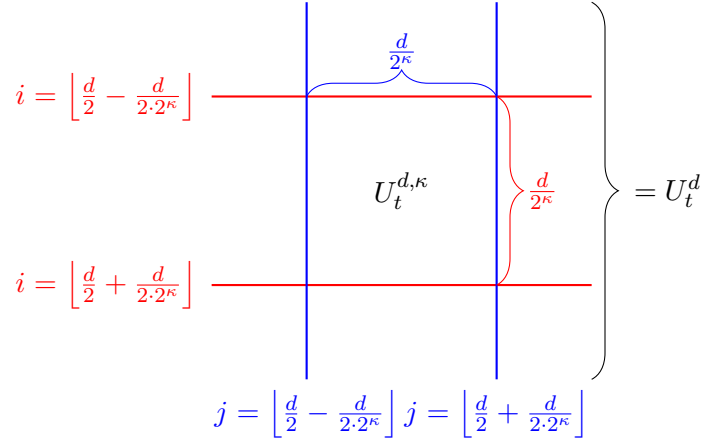
\begin{figure}%
\centering
\begin{tikzpicture}
    \node[kappaNode1] (center) {$U_t^{d,\kappa}$};
    \draw[kappaHorLine1,name path=horUpper]
        ($(center)+(-2.5,1.25)$) node[kappaNode1, color = red, anchor=east] (HUL)
            {$i=\left\lfloor \frac{d}{2}-\frac{d}{2 \cdot 2^\kappa}\right\rfloor$} --
        ($(center)+(2.5,1.25)$) node[kappaNode1] (HUR) {};
    \draw[kappaHorLine1,name path=horLower]
        ($(center)+(-2.5,-1.25)$)  node[kappaNode1, color = red, anchor=east] (HLL)
            {$i=\left\lfloor \frac{d}{2}+\frac{d}{2 \cdot 2^\kappa}\right\rfloor$} --
        ($(center)+(2.5,-1.25)$) node[kappaNode1] (HLR) {};
    \draw[kappaVertLine1,name path=vertLeft]
        ($(center)+(-1.25,-2.5)$)  node[kappaNode1, color = blue, anchor=north] (VLL)
            {$j=\left\lfloor \frac{d}{2}-\frac{d}{2 \cdot 2^\kappa}\right\rfloor$} --
        ($(center)+(-1.25,2.5)$) node[kappaNode1] (VUL) {};
    \draw[kappaVertLine1,name path=vertRight]
        ($(center)+(1.25,-2.5)$) node[kappaNode1, color = blue, anchor=north] (VLR)
            {$j=\left\lfloor \frac{d}{2}+\frac{d}{2 \cdot 2^\kappa}\right\rfloor$}--
        ($(center)+(1.25,2.5)$) node[kappaNode1] (VUR)  {};
    \path[name intersections={of=horUpper and vertLeft,by=UL}];
    \path[name intersections={of=horUpper and vertRight,by=UR}];
    \path[name intersections={of=horLower and vertLeft,by=LL}];
    \path[name intersections={of=horLower and vertRight,by=LR}];
    \draw [red,decorate,decoration={brace,amplitude=10pt},xshift=-4pt,yshift=0pt]
        (UR) -- (LR) node [red, midway,xshift=0.6cm] {$\frac{d}{2^\kappa}$};
    \draw [black,decorate,decoration={brace,amplitude=10pt},xshift=-4pt,yshift=0pt]
        ($(center)+(2.5,2.5)$) -- ($(center)+(2.5,-2.5)$)
            node [black, midway,xshift=1cm] {$=U_t^d$};
    \draw [blue,decorate,decoration={brace,amplitude=10pt},xshift=-4pt,yshift=0pt]
        (UL) -- (UR) node [blue, midway,yshift=0.6cm] {$\frac{d}{2^\kappa}$};
\end{tikzpicture}%
\caption{Graphical representation of $U_t^{d,\kappa}$ compared to $U_t^d$ for the error analysis to disregard boundary effects.}%
\label{fig:kappa}%
\end{figure}

For our error analysis we will use the following three norms:
\begin{align}
    \meanErr_t^{d,\kappa}&\coloneqq
    \frac{1}{M}\sum_{m=1}^{M}{
        \abs{U_{t,m}^{\text{ref},d,\kappa}-U_{t,m}^{\text{app},d,\kappa}}
    }\in \R^{\left\lfloor \frac{d}{2^\kappa}\right\rfloor {\color{blue}\times} \left\lfloor \frac{d}{2^\kappa}\right\rfloor}\\
    \avgMeanErr_t^{d,\kappa}&\coloneqq
    \frac{1}{\abs{I^\kappa}^2}
        \sum_{i,j}{
          \big(  \meanErr_t^{d,\kappa} \big)_{ij}
        }
    \in\R\\
    \Err_t^{d,\kappa}&\coloneqq
    \frac{1}{M}\sum_{m=1}^{M}{
        \frac{
            \norm{
                U_{t,m}^{\text{ref},d,\kappa}-
                U_{t,m}^{\text{app},d,\kappa}
            }_{F}
        }{
            \norm{
                U_{t,m}^{\text{ref},d,\kappa}
            }_{F}
        } \in \R
    },
\label{eq:meanErr}
\end{align}
where $\norm{\cdot}_F$ denotes the Frobenius norm.
The first one is a matrix consisting of a mean absolute error between the reference solution and the approximation for each point in the grid $I^\kappa \times I^\kappa$ by taking the mean over all trajectories.
The second error is the average of the first error for the corresponding grid indicated by $\kappa$. The larger $\kappa$ the further away we are from the spatial boundary,
as illustrated in Figure \ref{fig:kappa}.
The third error is the mean of a relative error between the reference solution and the approximation taking all points in the region corresponding to $\kappa$ into account. It will serve as our main error norm in this paper.

Regarding the Lebesgue discretization $\LebesgueDelta$ we have observed that
the computational times are more or less constant for $\LebesgueDelta\in \left[10^{-5},10^{-2}\right)$ and increase for lower $\LebesgueDelta$ significantly. This can be explained by the fact that the discretization is only used to compute Lebesgue-integrals of the form $\int_{0}^{t}{s^p W_s^q ds}$
. Therefore, the computation can be vectorized, which is very fast compared to the computation of the Magnus logarithm and the matrix-vector exponentials for $\LebesgueDelta\geq 10^{-5}$.

Since, there seems to be no computational drawback, we suggest to use a Lebesgue-discretization equal to $10^{-3}$ or $10^{-4}$ and fix it for all tests to $\LebesgueDelta = 10^{-4}$.

We have verified a linear behavior (with a slope less than one until the GPU is fully saturated) in the number of simulations $M$. In our experiments, we decided to use $M=100$ simulations and display always the average computational times for one simulation.

Also the computational effort with respect to the finite time horizon $T$ scales linearly for all methods. Thus, we use $T=1$ as our terminal time.

\begin{table}%
\caption{Notations for the numerical experiments.}
\centering
\begin{tabular}{ll}
    \euler & Euler-Maruyama scheme \eqref{eq:EulerMaruyama}\\
    \mOne & Iterated Magnus scheme of order 1\\
    \mTwo & Iterated Magnus scheme of order 2\\
    \mThree & Iterated Magnus scheme of order 3\\
    M & Number of simulations\\
        d & Number of grid points in $\mathbb{X}$ and $\mathbb{V}$\\
    $\stepSize$ & step-size of \euler or Magnus\\
    $\LebesgueDelta$ & discretization of the Lebesgue-integrals for Magnus\\
    M2, x & Magnus order 2 with step-size $\stepSize = x$\\
        M3, x & Magnus order 3 with step-size $\stepSize = x$\\
    E, x & \euler with step-size $\stepSize = x$
\end{tabular}
\label{tab:notations}
\end{table}
Throughout the experiments we will use the notation in Table \ref{tab:notations}.
We used for the calculations \matlab with the \matlabParalleltoolbox
running on \OS, on a machine with the following specifications: processor
\CPU, \RAM and a \GPU.


The next subsection is structured as follows: First, we derive the explicit solution of the Langevin equation. Then, we discuss the impact of the number of intervals $\stepSize$ for the iterated Magnus scheme regarding computational times and errors. Next, we look at the boundary effects over time. This is followed by a comparison of the Magnus scheme with the Euler-Maruyama scheme with different sizes of the space grid.

\subsection{The Magnus expansion for the stochastic Langevin equation with constant coefficients}\label{sec:stochasticLangevin}
In this subsection, we apply the Magnus expansion to the stochastic Langevin equation. For further details and a solution theory in Hölder spaces under the weak Hörmander condition we refer the reader to
\cite{Pascucci2022}.

In the constant coefficient case, the Langevin SPDE can be recovered from 
\eqref{eq:generalSPDE} setting 
\begin{align}
    h\equiv f^v\equiv g^{xx}\equiv g^{xv}\equiv \sigma \equiv \sigma^{x}\equiv 0,\quad
    f_x(x,v) \coloneqq -v, \quad g^{vv}\equiv a \in \R_{>0}, \quad \sigma^v(x,v)\equiv \sigma \in \R_{>0}.
    \label{eq:coeffsLangevinConst}
\end{align}
In this special case, there exists an explicit fundamental solution $\Gamma$ for $0<\sigma < \sqrt{a}$ (cf. \cite[p.~4 Proposition 1.1.]{Pascucci2022}), which is given by
\begin{align*}
    \Gamma\left(t,z;0,\zeta\right)&\coloneqq
    \Gamma_0\left(t,z-m_t(\zeta)\right),\\
    \Gamma_0\big(t,(x,v)\big)&\coloneqq
    \frac{
        \sqrt{3}
    }{
        \pi t^2 (a-\sigma^2)
    }
    \exp\left(
        -\frac{2}{a-\sigma^2}
            \left(
                \frac{v^2}{t}-\frac{3vx}{t^2}+\frac{3x^2}{t^3}
            \right)
    \right)
\end{align*}
where $\zeta \coloneqq \left(\xi,\eta\right)$ is the initial point and
\begin{align*}
    m_t(\zeta)\coloneqq
    \left(
        \begin{array}[c]{c}
            \xi + t \eta - \sigma \int_{0}^{t}{W_s ds}\\
            \eta - \sigma W_t
        \end{array}
    \right).
\end{align*}
Having the fundamental solution, we can solve the Cauchy-problem by integrating against the initial datum, i.e.
\begin{align*}
    u_t(z)= \int_{\R^2}{\Gamma(t,
    z;0,
    \zeta) \phi(
    \zeta) d\zeta, \qquad z=(x,v). 
    }
\end{align*}
To get an explicit solution (up to the stochastic integral $\int_{0}^{t}{W_s ds}$) for the double integral we will choose $\phi$ to be Gaussian, i.e.
\begin{align}\label{eq:initialdatum}
    \phi\left(\xi,\eta\right)\coloneqq
    \exp\left(
        -\frac{\left(\xi^2+\eta^2\right)}{2}
    \right).
\end{align}
The formula for the exact solution is lengthy and its specific form is not instructive for the following experiments, therefore we decided to exclude it from this presentation. The interested reader can find it in the corresponding \matlab code, which is publicly available.

Having an exact benchmark solution we will now perform some numerical tests to judge the performance of the iterated Magnus scheme.
Henceforth, the parameters for the stochastic Langevin equation will be $a=1.1$ and $\sigma=\frac{1}{\sqrt{10}}$, so that $a-\sigma^2=1>0$.

\paragraph*{Computational effort and errors with respect to the number of iterations.}
For this experiment we fix the number of grid points in each space grid to $d=200$ but vary
step-size of the Magnus scheme $\stepSize$.
In Figure \ref{fig:StepTest} we can see the corresponding results. The left y-axis shows the average computational times for one simulation in a log scale and the right y-axis the mean relative errors $\Err_T^{d,4}$ also in a log scale. The computational times (in seconds) of \mTwo are depicted in light blue and of \mThree in dark blue. Moreover, the mean relative errors for \mTwo are orange and for \mThree red.
\begin{figure}[H]%
\centering
\includegraphics[width=\columnwidth]{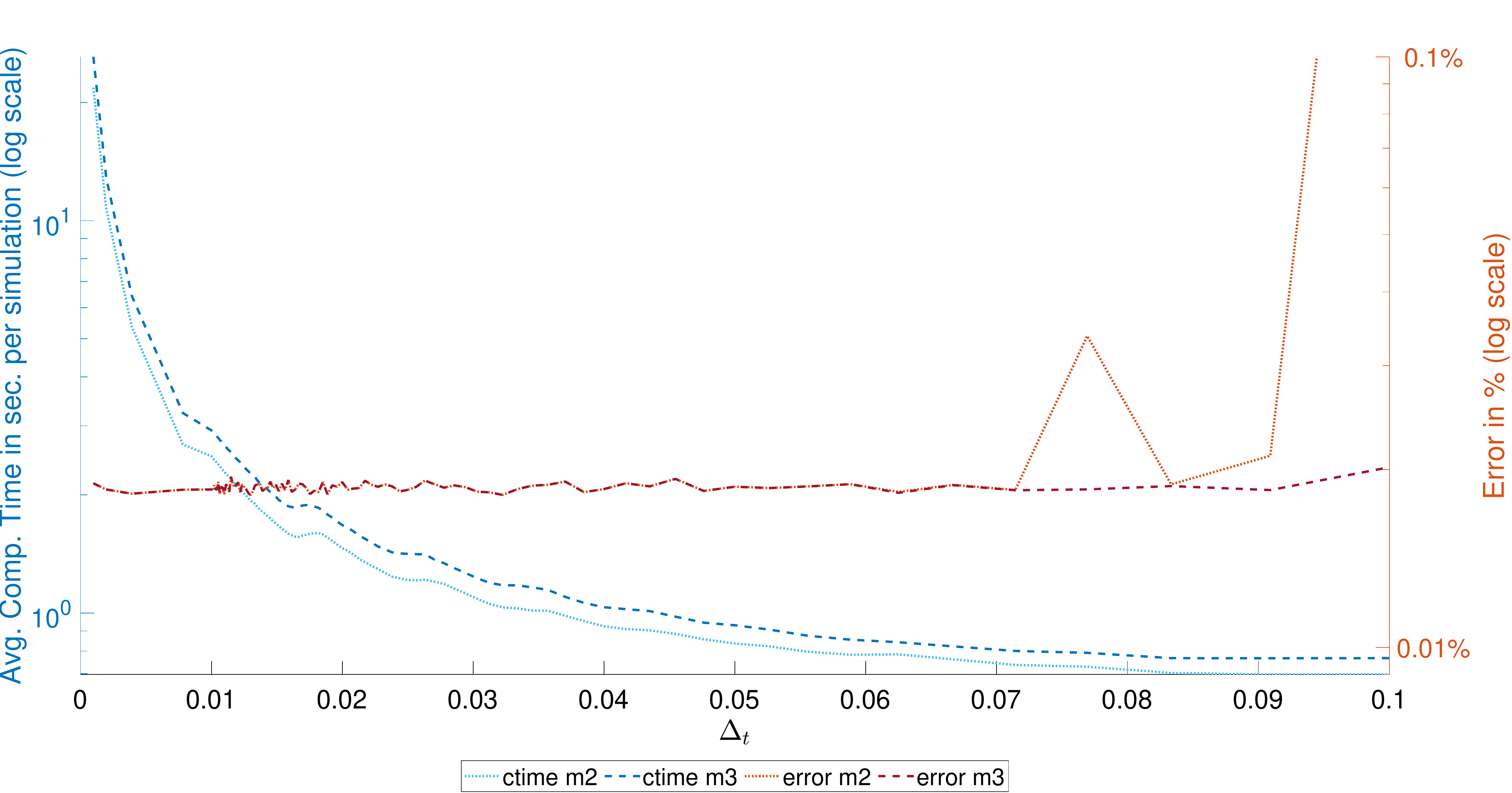}%
\caption{Constant coefficients as in \eqref{eq:coeffsLangevinConst}: Computational times and errors of the Magnus expansion for varying number step-size $\stepSize$ with fixed spatial dimension $d=200$.}%
\label{fig:StepTest}%
\end{figure}
We can see that the mean relative errors start to fluctuate for a step-size larger than
0.065 for \mTwo and see an explosion of \mTwo around 0.085.
For \mThree this phenomenon starts outside of the picture. The fluctuations begin right after $\stepSize=0.1$ and an explosion can be seen after $\stepSize=0.33$.
The explosions for large step-sizes are not surprising, since the step-size is determined by the underlying stopping times for the convergence of the Magnus scheme. The fluctuations beforehand are most likely due to an interplay between error propagations due to larger step-sizes and the necessary Taylor-terms in \verb+expmv+, and it indicates that \mThree is more stable than \mTwo. Therefore, experiment indicates that any step-size less than 0.05, 0.1 for $d=200$ is well within the convergence radius of \mTwo, \mThree, respectively, and yields stable results.

For other spatial dimensions $d$ this breaking point might be different. Moreover, we can see that the computational time increases more and more for smaller step-size, while the error for both methods stays almost constant and close to each other.

This suggests that one should choose the step-size as large as possible for the iterated Magnus scheme to gain the maximal performance. However, being too greedy will lead to blow-ups of some trajectories.

Also as a side note, usually with increasing spatial dimension $d$, one has to choose a smaller time step-size for the Magnus methods as well: this will be shown in Figure \ref{fig:dim100}--\ref{fig:dim300}.
\FloatBarrier

\paragraph*{Mean errors and boundary effects over time.}
For this experiment we fix the grid points in each space grid to $d=300$.

In Figure \ref{fig:errorLangevinConst} we can see the mean absolute errors of the entire spatial grid as a two-dimensional plot. A deep blue color indicates a small error and a bright yellow color an error up to $10^{-2}$. The black rectangle is the corresponding region for $\kappa=1$. The black number within the rectangle is the average mean absolute error of the corresponding region. The picture on the left-hand side is the area of errors
at $t=0.25$ and on the right-hand side at $t=1$.
%

\begin{figure}%
\centering
\includegraphics[width=.49\columnwidth,page=10]{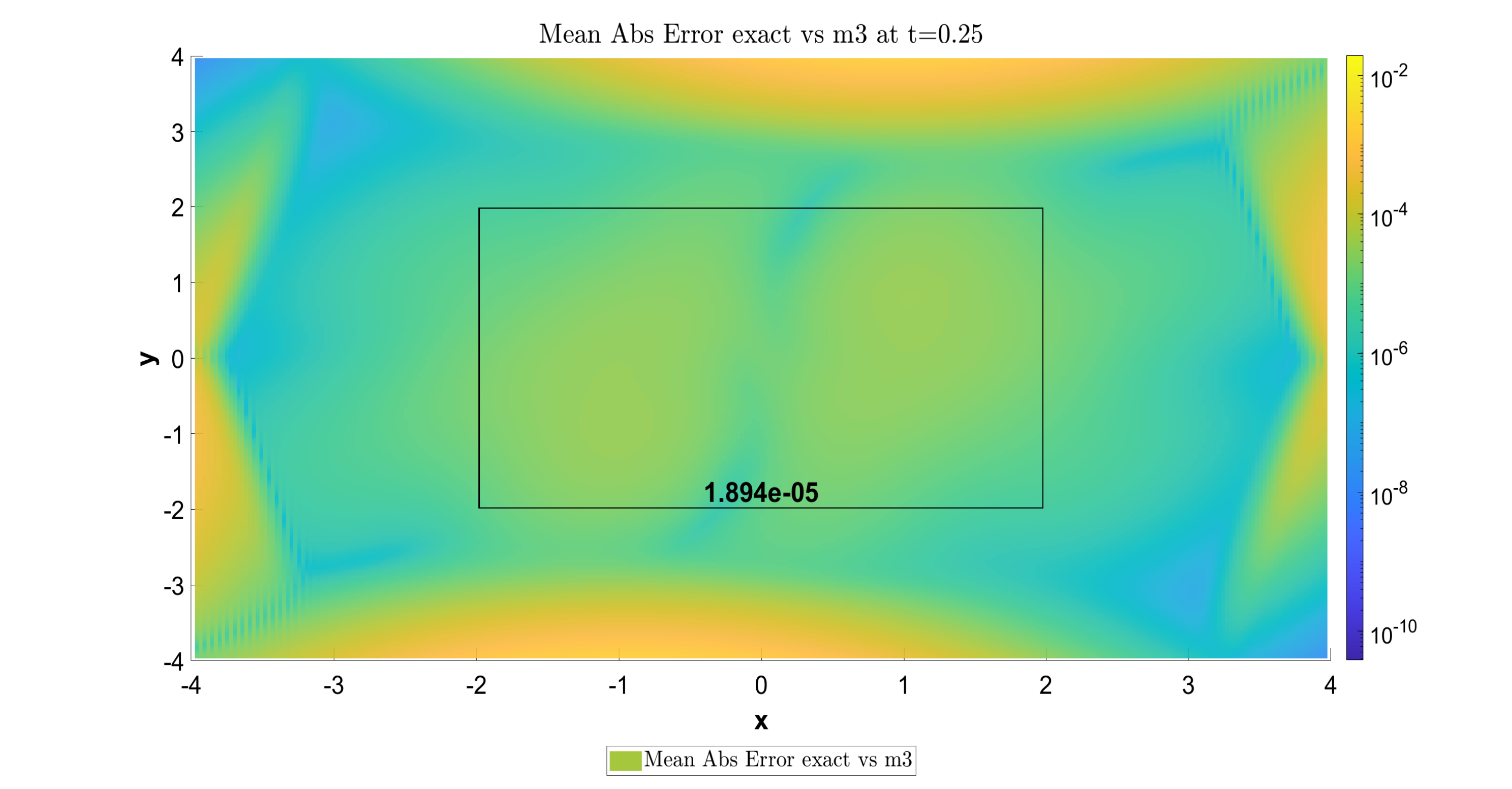}
\includegraphics[width=.49\columnwidth,page=40]{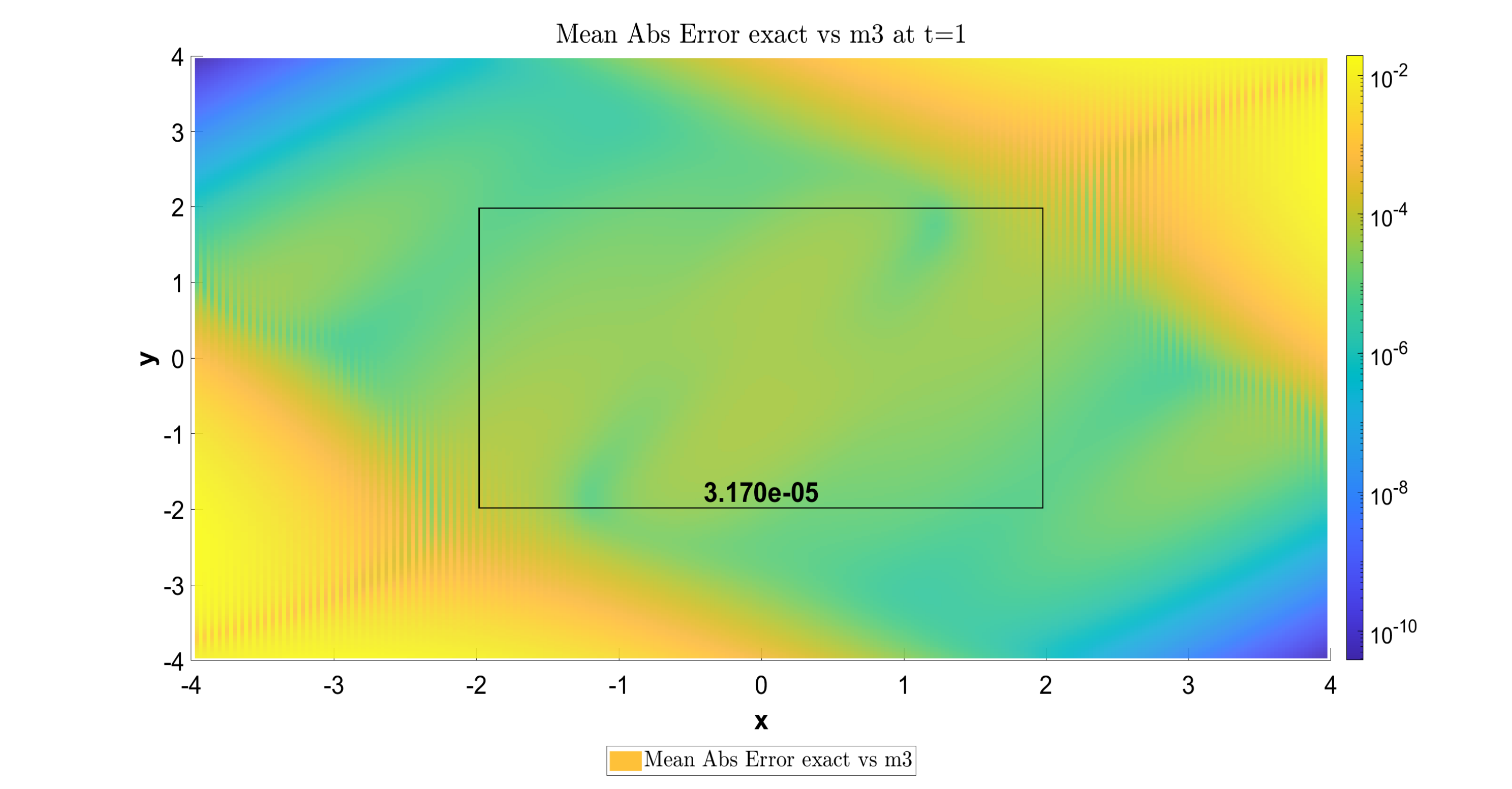}%
\caption{Constant coefficients as in \eqref{eq:coeffsLangevinConst}: Absolute Errors of \mThree compared to \exact using $d=300$ grid points at $t=0.25$ (left) and $t=1$ (right).}%
\label{fig:errorLangevinConst}%
\end{figure}
We can see that the errors in the upper right and lower left corners are significantly increasing over time. To explain this, one should note that the Langevin equation with this specific initial datum looks like a two-dimensional Gaussian at first and its shape changes on the diagonal from the lower left corner to the upper right corner over time more than on the other diagonal. Therefore, the cut-off region is getting too small for larger times leading to boundary effects in the error plots. This also explains why the upper left and lower right corner remain a stable small error. In the center of the error plots we can see an increasing error over time. If this is due to the boundary effects, error propagation due to the iterated scheme, the error due to the order 3 truncation or the algorithm used for the matrix-vector exponential is not apparent in this illustration, we suspect a mixture of all of them.

\paragraph*{Comparison to the Euler-Maruyama scheme}
For this experiment, we will compare different choices of parameters for both the Magnus scheme and Euler-Maruyama scheme. There are essentially two major parameters contributing to the possible accuracy. One is the time step-size of the individual schemes and the other one the space discretization. Hence, we compare Euler and Magnus methods with different time step-sizes for different space discretizations $d=100,200,300$ to increase the level of accuracy.
In the Figures \ref{fig:dim100}, \ref{fig:dim200} and \ref{fig:dim300} the left y-axis shows the average computational times in a log scale and the right y-axis the mean relative errors $\Err_T^{d,4}$ also in a linear scale. The computational times (in seconds) are depicted in the left blue columns and the mean relative errors in the red right columns for each method.

As mentioned in Table \ref{tab:notations}, ``E, x'' denotes Euler with step-size
$\stepSize=x$ and ``M2, x'', ``M3, x'' denotes Magnus with step-size $\stepSize=x$ for order 2 and 3, respectively. In Figure \ref{fig:dim100} we compare the errors and computational times of the methods with spatial dimension $d=100$, in Figure \ref{fig:dim200} with $d=200$ and in Figure \ref{fig:dim300} with $d=300$.

\begin{figure}%
\centering
\includegraphics[width=\columnwidth]{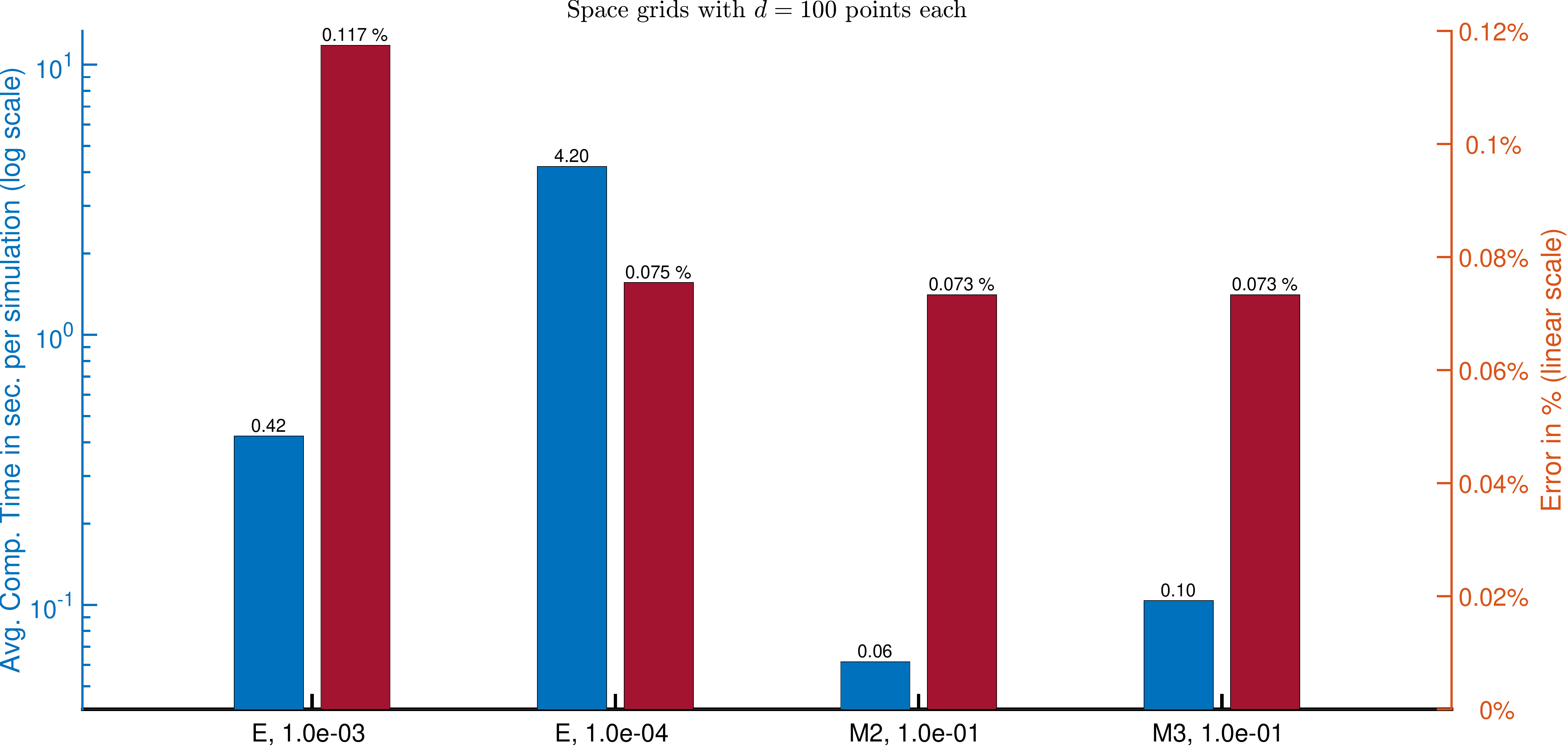}%
\caption{Constant coefficients as in \eqref{eq:coeffsLangevinConst}: Computational times and errors of the Magnus expansion and Euler scheme for $d=100$.}%
\label{fig:dim100}%
\end{figure}
Let us focus on Figure \ref{fig:dim100} with $d=100$. We can see that four different methods are compared: the Euler method with step-size $\stepSize=10^{-3}$ and $\stepSize=10^{-4}$, as well as the Magnus method with step-size $\stepSize=0.1$ of order 2 and order 3. It is notable that the Euler method with step-size $\stepSize=10^{-4}$ and the Magnus methods perform almost the same with respect to the error. The Euler method with step-size $\stepSize=10^{-3}$ has roughly double the error of the method with step-size $\stepSize=10^{-4}$ but is ten times faster. Overall, the Magnus methods were the fastest methods. The Magnus method of order two, three is $70$, $42$ times, respectively, faster than Euler method with step-size $\stepSize=10^{-4}$ and has a slightly better accuracy.

\begin{figure}%
\centering
\includegraphics[width=\columnwidth]{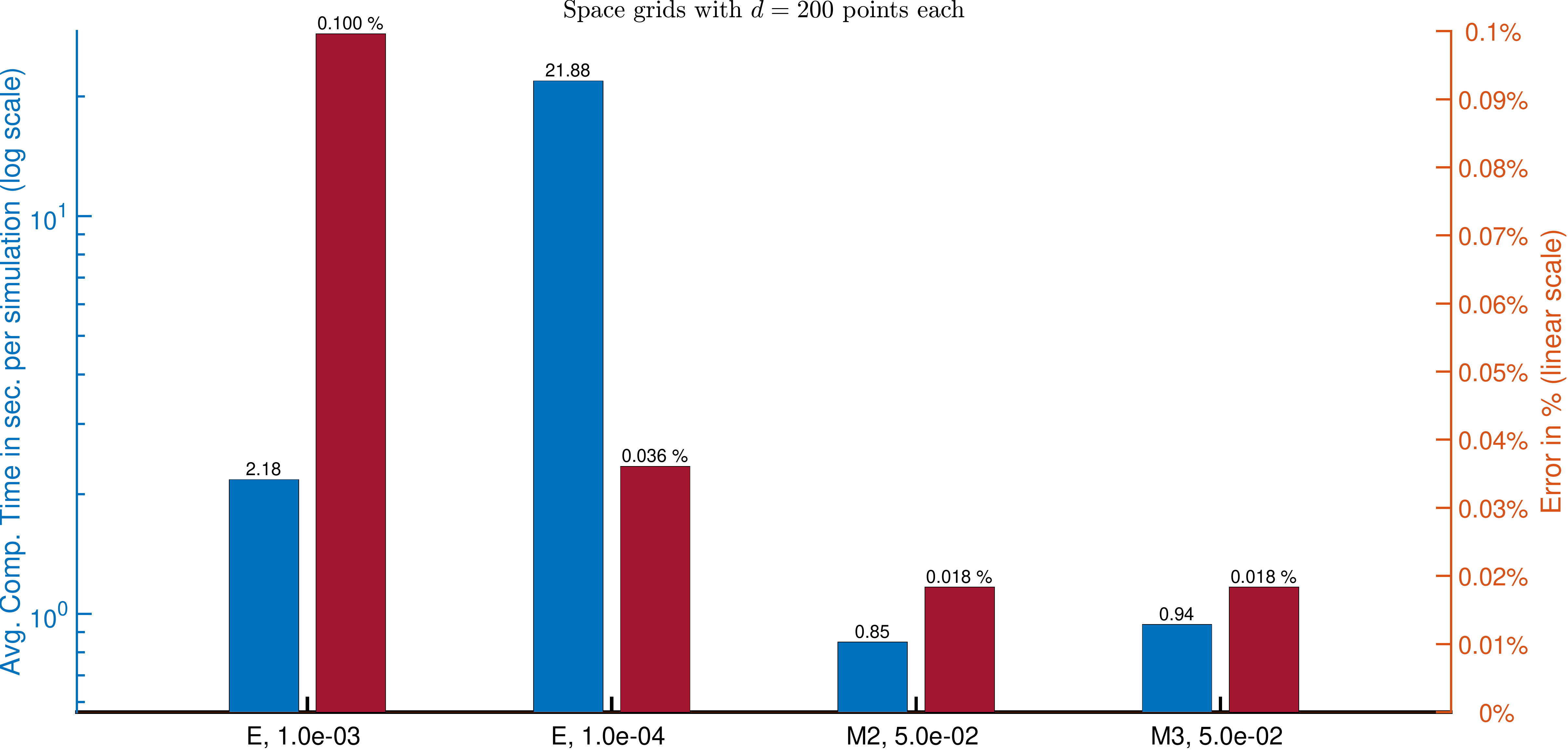}%
\caption{Constant coefficients as in \eqref{eq:coeffsLangevinConst}: Computational times and errors of the Magnus expansion and Euler scheme for $d=200$.}%
\label{fig:dim200}%
\end{figure}
Now, let us consider Figure \ref{fig:dim200} with $d=200$. Again, we can see that two Euler methods and two Magnus methods are compared but this time we have a step-size $\stepSize=0.05$ for the Magnus methods. Similarly, to Figure \ref{fig:dim100}, we can see that the Euler method with step-size $\stepSize=10^{-3}$ performed worst and the Magnus methods best in terms of accuracy. However, this time the Euler method with step-size $\stepSize=10^{-4}$ has twice the error compared to the Magnus methods and is
$25$, $23$ times slower than the Magnus method with order 2, 3, respectively.

%

\begin{figure}%
\centering
\includegraphics[width=\columnwidth]{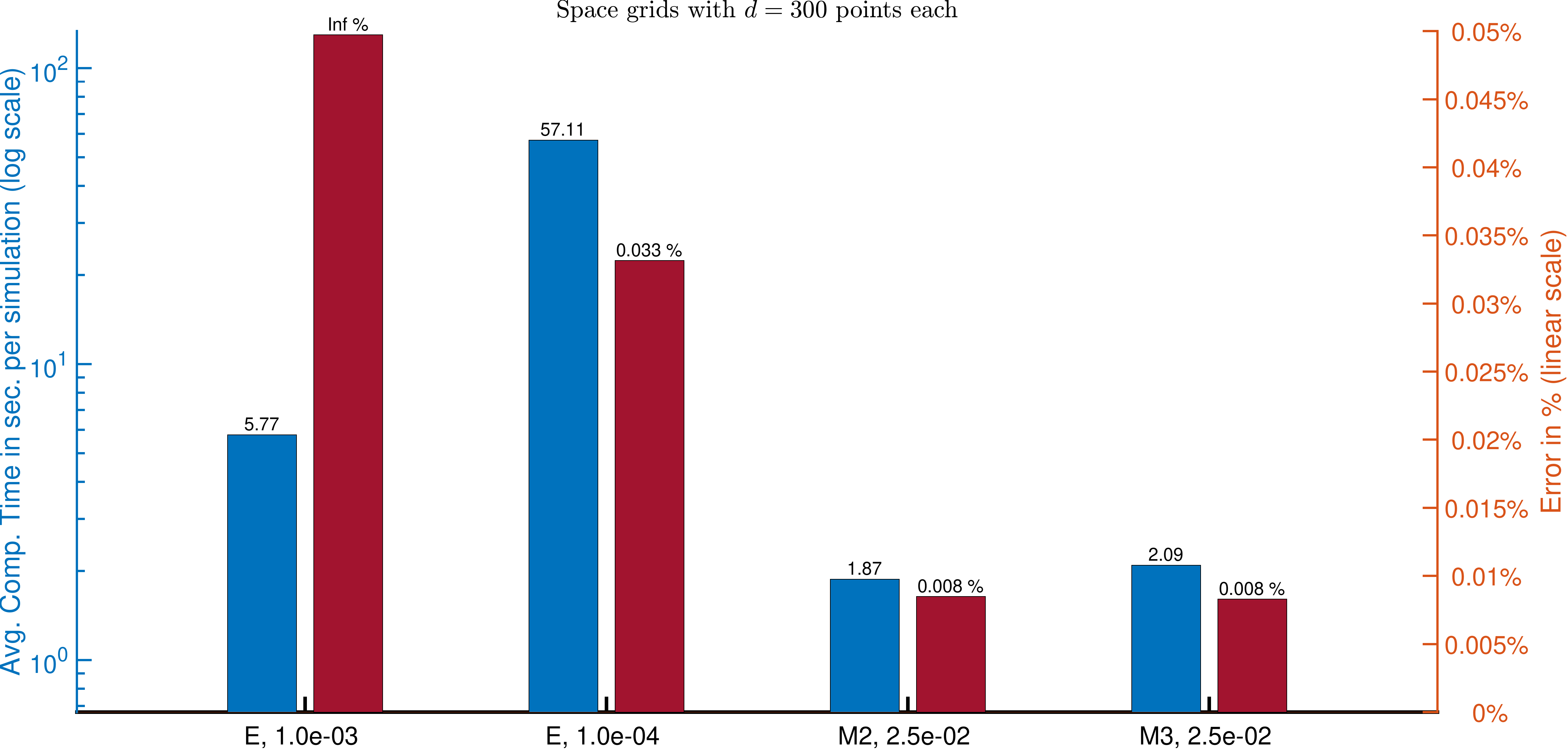}%
\caption{Constant coefficients as in \eqref{eq:coeffsLangevinConst}: Computational times and errors of the Magnus expansion and Euler scheme for $d=300$.}%
\label{fig:dim300}%
\end{figure}
In Figure \ref{fig:dim300} with $d=300$ the Euler method with step-size $\stepSize=10^{-3}$ is exploding, since its stability criterion is violated and its errors is $\infty$.
Therefore, we compare the Euler method with step-size $\stepSize=10^{-4}$ to the Magnus method with with step-size $\stepSize=0.025$ with order two and three. This time the Euler method is four times worse in terms of accuracy and $30$, $27$ times slower than Magnus with order two, three, respectively. We also performed tests with an Euler method using step-size $\stepSize=10^{-5}$. Its accuracy was still slightly worse compared to the Magnus methods and its computational time ten times slower than the Euler scheme in the figure. This results in a speed-up of order 250 of the Magnus method compared to an Euler scheme with similar accuracy.

Overall,
from these observation it is clear that an Euler scheme with a fine time-discretization is essential to make it comparable to the iterated Magnus scheme in terms of accuracy.
Moreover, increasing the number of grid points is leading to less accurate
errors using the Euler method with step-size $\stepSize=10^{-4}$ compared to the Magnus schemes with corresponding step-sizes, while the Magnus methods remain roughly $30$ times faster in all tests.

\begin{remark}\label{rem:stepSizeControl}%
    If we have a close look at all tests from above we can see that all of them share a common feature, namely for reasonable parameters \mTwo and \mThree were always close.
    Therefore, this leads to a natural step-size control in time by comparing the results of Magnus order 2 to order 3.
    If they are closer than a given tolerance then the step-size is small enough, otherwise make it smaller by a given factor.

    With this method, the computation of the Magnus logarithms up to order 3 can be reused for Magnus order 2. However, two matrix-vector exponentials for each trajectory are necessary to determine if the time-step is rejected. For implementations with a lot of trajectories, one can think about using less randomly chosen trajectories to determine the correct step-size to increase the overall performance.
\end{remark}

\begin{remark}\label{rem:determLangevin}%
    The Magnus expansion holds an advantage over all other finite-difference method in the deterministic case. Inspecting the approximation formulas in the case
    $A\equiv 0$ reveals immediately that the Magnus expansion is exact, at order $1$, up to the initial space discretization for $x$ and $v$, meaning that its accuracy is far more superior than e.g. explicit and implicit Euler-schemes.
\end{remark}

\begin{remark}\label{rem:piecewise}%
   We would like to point out, that it is straightforward to use the iterated Magnus scheme in the case of deterministic piecewise constant coefficients in time by partitioning the time interval according to the piecewise definition. Then on each of these intervals the Magnus expansion formulas for constant coefficients hold.
\end{remark}
\subsection{The Magnus expansion for the stochastic Langevin equation with variable coefficients}\label{sec:stochasticLangevinVar}
In this brief subsection, we will perform some tests in the case of variable coefficients. In particular, we choose bounded, smooth coefficients of the form
\begin{align}
    &h\equiv f^v\equiv g^{xx}\equiv g^{xv}\equiv \sigma \equiv \sigma^{x}\equiv 0,\\&
    f_x(x,v) \coloneqq -v, \quad g^{vv}(x,v) =  a \left(1+\frac{1}{x^2+1}\right), \quad \sigma^v(x,v)\equiv \sigma \sqrt{1+\frac{1}{x^2+1}},
    \label{eq:coeffsLangevinVar}
\end{align}
with $a,\sigma\in\R_{>0}$, as above, satisfying $g^{vv}(x,v) - \left(\sigma^v(x,v)\right)^2 \equiv 1 >0$, as in the constant coefficient case. We will also use the same initial condition as in \eqref{eq:initialdatum}. 


Analog to Figure \ref{fig:StepTest}, we show in Figure \ref{fig:StepTestVar} the case of varying step-sizes for the Magnus method with fixed spatial discretization $d=200$.
The average computational times of \mTwo and \mThree in seconds, per simulation, are again depicted in light blue and dark blue, respectively. The errors are this time with respect to the Euler method with $\stepSize=10^{-4}$, since an exact solution is not available, and again illustrated as orange for \mTwo and red for \mThree.

\begin{figure}[H]%
\centering
\includegraphics[width=\columnwidth]{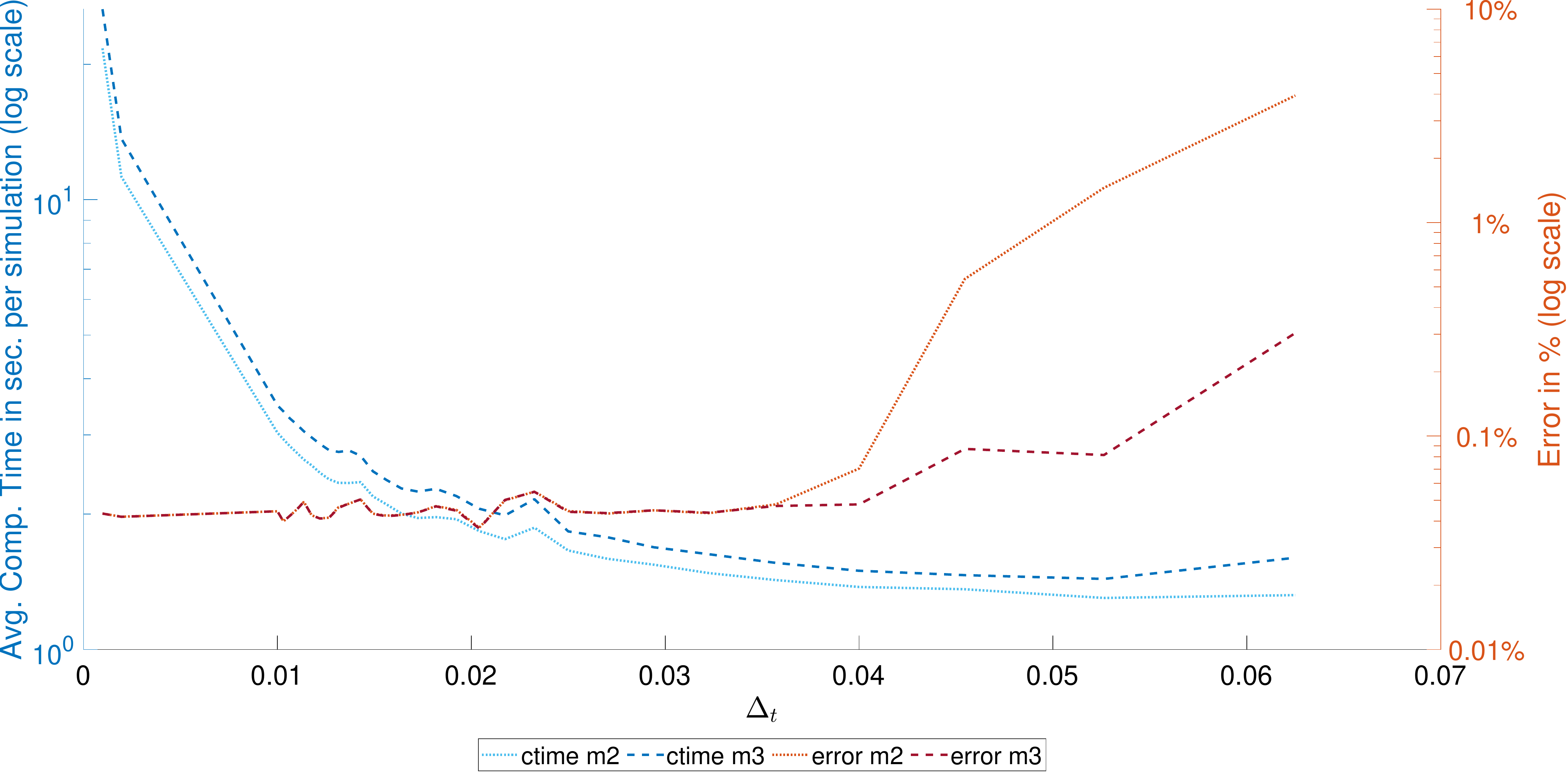}%
\caption{Variable coefficients as in \eqref{eq:coeffsLangevinVar}: Computational times and errors compared to Euler with $\stepSize=10^{-4}$ of the Magnus expansion for varying number step-size $\stepSize$ with fixed spatial dimension $d=200$.}%
\label{fig:StepTestVar}%
\end{figure}
The computational times look similar to the constant coefficient case and we can notice a slight increase of the computational times for \mThree at around $0.055$. This behavior is the same in the constant coefficient case, but can only be seen at step-sizes larger than $\stepSize=0.1$. Therefore it is not part of the previous figure. This can be explained by the necessity to use more Taylor terms in the matrix-vector exponential to keep the errors small.

Moreover, we can see that the mean relative errors start to increase for a step-size larger than 0.035 for \mTwo and 0.045 for \mThree. This happens earlier than in the constant coefficient case and is the expected behavior.

In our next experiment, similarly to Figure \ref{fig:dim100}--\ref{fig:dim300}, we show in Table \ref{tab:errorsVar} the mean relative errors and computational times for $d=100,200,300$. This time we use an Euler method with $\stepSize=10^{-5}$ as our reference solution and compare an Euler method
with $\stepSize=10^{-4}$, as well as \mTwo and \mThree to it.
The results are given in Table \ref{tab:errorsVar}.
\begin{table}%
\caption{Variable coefficients as in \eqref{eq:coeffsLangevinVar}: Computational times and errors compared to Euler with $\stepSize=10^{-5}$ of the Magnus expansion for different spatial dimension $d=100,200,300$.}
\centering
\begin{tabular}{lcc}
    Method & Mean Rel. Error (in \%) & Comp. Time (in sec./simulation)\\
    \toprule
    \multicolumn{3}{c}{$d=100$}\\
    \midrule
    \euler, $\stepSize=10^{-3}$ & 0.147\,\% & 0.43\\
    \euler, $\stepSize=10^{-4}$ & 0.047\,\% & 4.24\\
    \euler, $\stepSize=10^{-5}$ & -- & 42.24 \\
    \mTwo, $\stepSize=0.05$ & 0.015\,\% & 0.13 \\
    \mThree, $\stepSize=0.05$ & 0.014\,\% & 0.24 \\
    \midrule
    \multicolumn{3}{c}{$d=200$}\\
    \midrule
    \euler, $\stepSize=10^{-3}$ & $\infty$\,\% & 2.03\\
    \euler, $\stepSize=10^{-4}$ & 0.045\,\% & 21.02\\
    \euler, $\stepSize=10^{-5}$ & -- & 209.02 \\
    \mTwo, $\stepSize=0.025$ & 0.013\,\% & 1.66 \\
    \mThree, $\stepSize=0.025$ & 0.012\,\% & 1.83 \\
    \midrule
    \multicolumn{3}{c}{$d=300$}\\
    \midrule
    \euler, $\stepSize=10^{-3}$ & $\infty$\,\% & 5.79\\
    \euler, $\stepSize=10^{-4}$ & 0.041\,\% & 58.41\\
    \euler, $\stepSize=10^{-5}$ & -- & 582.88\\
    \mTwo, $\stepSize=0.01$ & 0.014\,\% & 4.28\\
    \mThree, $\stepSize=0.01$ & 0.014\,\% & 5.0\\
\end{tabular}
\label{tab:errorsVar}
\end{table}

We can see that computational times are roughly twice higher for the Magnus scheme than in the constant coefficient case. The Euler scheme did not suffer from a performance decrease. As aforementioned this increase of computational effort can be attributed to the decreased sparsity. Therefore, the Magnus scheme is only 12 times faster than the Euler scheme.
Regarding the accuracy, we can see in all cases that the Magnus expansion is more accurate than \euler with $\stepSize=10^{-4}$.

All in all we come to the same conclusion as in the constant coefficient case.
\section{Conclusions}\label{sec:conclusion}
We have seen how to derive the Itô-stochastic Magnus expansion for SDEs with constant
matrices and used it to solve two-dimensional SPDEs with a given initial datum numerically. We derived an approximation scheme for a generic parabolic SPDE with two space variables, then tested it on a special class of Langevin-type SPDEs. 
The scheme has an excellent accuracy and its advantage in terms of computational effort
excels for higher spatial resolution. For instance, in the case of constant coefficients to have roughly the same accuracy for $d = 400$ we need an
Euler scheme with $\stepSize = 10^{-5}$ taking approximately $21$ minutes, in average per trajectory, while Magnus order 3 takes only $4.62$
seconds using $\stepSize = 0.01$, $\LebesgueDelta = 10^{-4}$. 
This is a speed-up by a factor 280 just using one GPU while sparsity ensures an almost equal memory demand. Adding multiple GPUs will lead to further improvements with respect to computational times.
\appendix
\section{Magnus expansion formulas for constant coefficients}\label{sec:MagnusFormulas}
In order to make the paper self-contained, in this section, we demonstrate how to heuristically derive the Itô-stochastic Magnus expansion formula for the equation
\begin{align}
    dX_t = B X_t dt + A X_t dW_t, \quad
    X_0=I_d, \quad
    \label{eq:MagnusFormulasCorrelated}
\end{align}
where the coefficients $B,A \in \R^{d\times d}$ are constant matrices and $W$ is a standard one-dimensional Brownian motion.

We make the ansatz $X_t=\exp\left(Y_t\right)$, where
\begin{align*}
    Y_t=
    \int_{0}^{t}{\mu\left(s,Y_s\right) ds}+
    \int_{0}^{t}{\sigma\left(s,Y_s\right) dW_s}, \quad Y_0=0_{\R^{d\times d}}.
\end{align*}
By Itô's formula (cf. \cite[p.~8 Lemma 1 and p.~9 Proposition 1]{KPP2021}) we have
 we have
\begin{align*}
    dX_t &= B X_t + A X_t dW_t\\
           &= d\exp\left(Y_t\right)\\
             &=
    \begin{aligned}[t]\arraycolsep=0pt
    &\biggl(
    \mathcal{L}_{Y_t}\left(\mu\left(s,Y_s\right)\right)
    +\frac{1}{2}
    \mathcal{Q}_{Y_t}\left(\sigma\left(t,Y_t\right),\sigma\left(t,Y_t\right)\right)
    \biggr) \exp\left(Y_t\right) dt\\&+
    \mathcal{L}_{Y_t}\left(\sigma\left(t,Y_t\right)\right)\exp\left(Y_t\right) dW_t,
    \end{aligned}
\end{align*}
where the operators $\mathcal{L}$ and $\mathcal{Q}$ associated to the first and second order derivative of the exponential map are given by
\begin{align*}
   \mathcal{L}_{Y} (M) &\coloneqq
    \sum_{n=0}^{\infty} \frac{1}{(n+1)!} \ad^n_{Y}(M),\\
		\mathcal{Q}_Y(M,N) &\coloneqq
		\sum_{n=0}^{\infty}{
        \sum_{m=0}^{\infty}{
            \frac{ \ad^n_{Y}(M)}{(n+1)!} \frac{ \ad^m_{Y} (N)}{(m+1)!}
        }
    }  +
    \sum_{n=0}^{\infty}{
        \sum_{m=0}^{\infty}{
            \frac{
                \comm{\ad^n_{Y}(N)}{\ad^m_{Y}(M)}
            }{
                (n+m+2)(n+1)!m!
            }
        }
    }.
\end{align*}
A comparison of coefficients yields
\begin{align*}
    B &\overset{!}{=}
    \mathcal{L}_{Y_t}\left(\mu\left(t,Y_s\right)\right)
    +\frac{1}{2}
    \mathcal{Q}_{Y_t}\left(\sigma\left(t,Y_t\right),\sigma\left(t,Y_t\right)\right)\\
    A &\overset{!}{=}
    \mathcal{L}_{Y_t}\left(\sigma\left(t,Y_t\right)\right).
\end{align*}
Now, Baker's lemma (cf. \cite[p.~10 Lemma 2]{KPP2021}) provides us with suitable conditions
and a series representation for the inverse of the first-order derivative operator $\mathcal{L}$, i.e.
\begin{align}
    \sigma\left(t,Y_t\right)&\equiv
    \sigma\left(Y_t\right)=
    \sum_{n=0}^{\infty}{
        \frac{\beta_n}{n!} \ad^n_{Y_t}\left(A\right)
    }\label{eq:sigmaMagnus}\\
    \mu\left(t,Y_t\right)&\equiv
    \mu\left(Y_t\right)=
    \begin{aligned}[t]\arraycolsep=0pt
        \sum_{k=0}^{\infty}
            \frac{\beta_k}{k!}
            \ad^k_{Y_t}\Biggl(
                B -
                \frac{1}{2}
                    \sum_{n=0}^{\infty}
                        \sum_{m=0}^{\infty}&
                            \frac{\ad^n_{Y_t}\left(\sigma\left(Y_t\right)\right)}{(n+1)!}
                            \frac{\ad^m_{Y_t}\left(\sigma\left(Y_t\right)\right)}{(m+1)!}
                            \\&+
                            \frac{
                                \left[
                                    \ad^n_{Y_t}\left(\sigma\left(Y_t\right)\right),
                                    \ad^m_{Y_t}\left(\sigma\left(Y_t\right)\right)
                                \right]
                            }{(n+m+2)(n+1)!m!}
            \Biggr),
    \end{aligned}
    \label{eq:muMagnus}
\end{align}
The Bernoulli numbers are denoted by $\beta_k$ and we recall that $\beta_0=1$, $\beta_1=-\frac{1}{2}$, $\beta_2=\frac{1}{6}$, $\beta_3=0$
and $\beta_4=-\frac{1}{30}$.

Now, we know the coefficients of $Y_t$ and solve the Itô-SDE by means of a stochastic Picard iteration, i.e. we start at $n=0$ with $Y_t^0 = 0_{\R^{d\times d}}$
and iterate for $n\geq 1$
\begin{align}
    Y_t^n=
    \int_{0}^{t}{
        \mu\left(s,Y_s^{n-1}\right)ds
    }+
    \int_{0}^{t}{
        \sigma\left(s,Y_s^{n-1}\right) dW_s
    }.
    \label{eq:PicardIteration}
\end{align}
In order to derive the Magnus expansion formulas we will introduce some bookkeeping parameters $\epsilon,\delta >0$ and substitute $A$ by $\epsilon A$, as well as
$B$ by $\delta B$. Henceforth, we will denote the $n$-th order Picard iteration with the substitution by $Y^{n,\epsilon,\delta}_t$.
\paragraph*{Order 1.}
Let us derive the first-order Magnus expansion, meaning we are interested in all terms with the first power of $\epsilon$ and $\delta$ in the first-order Picard iteration. Thus, we insert $Y^{0,\epsilon,\delta}_t\equiv 0$ into
\eqref{eq:PicardIteration}. Notice, that the zero matrix commutes with all matrices and therefore by definition of $\ad^0_{Y}(A)=A$ we have
\begin{align*}
    \sigma\left(Y^{0,\epsilon,\delta}_t\right)=\epsilon A.
\end{align*}
Inserting this into the formula for $\mu$ yields
\begin{align*}
    \mu\left(Y^{0,\epsilon,\delta}_t\right)=
    \delta B - \frac{1}{2} \epsilon^2 A^2,
\end{align*}
because $A$ commutes with itself as well.

Since, the Itô-correction term is of order $\epsilon^2$ it will not be part of the first-order Magnus expansion and we have
\begin{align*}
    Y^{1}_t = \int_{0}^{t}{B ds} + \int_{0}^{t}{A dW_s}
                    = B t + A W_t.
\end{align*}
\paragraph*{Order 2.} To derive the second-order Magnus expansion let us first think about how many terms we need at most in the infinite sums of \eqref{eq:sigmaMagnus} and
\eqref{eq:muMagnus} to disregard the majority of the high-order terms.

Let us consider \eqref{eq:sigmaMagnus} first and let $N=2$ denote the desired order of the Magnus expansion. The operator $\ad^0$ will at least result in a first order term in $\epsilon$ or $\delta$. The commutator $\ad^1_Y\left(X\right)=\left[Y,X\right]=YX-XY$ will at least have second order terms, because $X$ and $Y$ will at least have a first-order $\epsilon$ or $\delta$ coefficient. Now, the nested commutator
$\ad^2_Y\left(X\right)=\left[Y,\left[Y,X\right]\right]$ will have at least a third-order term for the same reason.
Therefore, to compute the $N$-th order of the Magnus expansion we only need to consider
the infinite sum in \eqref{eq:sigmaMagnus} up to $N-1$.
Now, let us consider \eqref{eq:muMagnus}. We will split our consideration into two parts by linearity and focus for the moment on
\begin{align*}
    \sum_{k=0}^{\infty}{
        \ad_{Y_t}^k\left(B\right)
    }.
\end{align*}
With exactly the same arguments as for \eqref{eq:sigmaMagnus} we only need to consider
$N-1$ terms of this infinite sum for the $N$-th order Magnus expansion.

Let us turn now to the rest, the Itô-correction term, of \eqref{eq:muMagnus}, and split it again in two parts, i.e. we consider
\begin{align*}
    \sum_{k=0}^{\infty}
        \ad^k_{Y_t}
        \left(
            \sum_{n=0}^{\infty}
            \sum_{m=0}^{\infty}
            \frac{\ad_{Y_t}^n\left(\sigma\left(Y_t\right)\right)}{(n+1)!}
            \frac{\ad_{Y_t}^m\left(\sigma\left(Y_t\right)\right)}{(m+1)!}
        \right).
\end{align*}
In the case $0=k=n=m$ we will get $\sigma^2$, which has at least order 2 terms. Hence, we will only need $k=1,\dots,N-2$. This also means, that $n$ and $m$ cannot exceed $N-2$.

Increasing $k$ and keeping $n=m=0$ will also increase the order due to the nested commutators as in the case for \eqref{eq:sigmaMagnus}. Therefore, the lowest $k$ can have the most terms resulting from $n$ and $m$. Vice versa, the higher $n$ and $m$ the lower $k$ must be to have small enough orders.

Taking all these considerations into account, we can use the following formula to cover all necessary terms up until order $N$: 
\begin{align}
    \sigma_N\left(Y_t\right)&\coloneqq
    \sum_{n=0}^{N-1}{
        \frac{\beta_n}{n!} \ad^n_{Y_t}\left(A\right)
    },\label{eq:sigmaMagnusCut}\\
    \mu_N\left(Y_t\right)&\coloneqq
        \sum_{k=0}^{N-1}
            \frac{\beta_k}{k!}
            \ad^k_{Y_t}\left(
                B
            \right)\\&\quad
        \begin{aligned}[t]\arraycolsep=0pt
        -\frac{1}{2}
        \sum_{k=0}^{N-2}
        \frac{\beta_k}{k!}
            \ad^k_{Y_t}\Biggl(&
                    \sum_{n=0}^{N-2-k}
                        \sum_{m=0}^{N-2-k}
                            \frac{\ad^n_{Y_t}\left(\sigma_{N-k}\left(Y_t\right)\right)}{(n+1)!}
                            \frac{\ad^m_{Y_t}\left(\sigma_{N-k}\left(Y_t\right)\right)}{(m+1)!}
                            \\&+
                            \frac{
                                \left[
                                    \ad^n_{Y_t}\left(\sigma_{N-k}\left(Y_t\right)\right),
                                    \ad^m_{Y_t}\left(\sigma_{N-k}\left(Y_t\right)\right)
                                \right]
                            }{(n+m+2)(n+1)!m!}
            \Biggr).
    \end{aligned}
    \label{eq:muMagnusCut}
\end{align}
There will still be a lot of higher-order terms in \eqref{eq:sigmaMagnusCut} and
\eqref{eq:muMagnusCut} but it will make our derivation of the formulas easier.

Now, we compute $\sigma_2\left(Y_t^1\right)$:
\begin{align*}
    \sigma_2\left(Y_t^{1,\epsilon,\delta}\right)&=
    \frac{1}{0!} \ad^0_{Y_t^{1,\epsilon,\delta}}\left(\epsilon A\right)-
    \frac{1}{2} \ad^1_{Y_t^{1,\epsilon,\delta}}\left(\epsilon A\right)\\&=
    \epsilon A-\frac{1}{2}\left[Y_t^{1,\epsilon,\delta},\epsilon A\right]\\&=
    \epsilon A-\frac{1}{2}\left[\delta B,\epsilon A\right] t -\frac{1}{2} \left[\epsilon A,\epsilon A\right]W_t\\&=
    \epsilon A-\frac{1}{2}\left[\delta B,\epsilon A\right] t.
\end{align*}
Next, we compute $\mu_2\left(Y_t^1\right)$:
\begin{align*}
    \hspace{1em}&\hspace{-1em}
    \mu_2\left(Y_t^{1,\epsilon,\delta}\right)=
    \frac{1}{0!} \ad^0_{Y_t^{1,\epsilon,\delta}}\left(\delta B\right)-
    \frac{1}{2} \ad^1_{Y_t^{1,\epsilon,\delta}}\left(\delta B\right)\\&\quad-
    \frac{1}{2}\Biggl(
        \frac{1}{0!}
            \biggl(
                \frac{\ad^0_{Y_t^{1,\epsilon,\delta}}\left(\sigma_2\left(Y_t^{1,\epsilon,\delta}\right)\right)}{1!}
                \frac{\ad^0_{Y_t^{1,\epsilon,\delta}}\left(\sigma_2\left(Y_t^{1,\epsilon,\delta}\right)\right)}{1!}
                \\&\qquad+
                \frac{\left[\ad^0_{Y_t^{1,\epsilon,\delta}}\left(\sigma_2\left(Y_t^{1,\epsilon,\delta}\right)\right),
								\ad^0_{Y_t^{1,\epsilon,\delta}}\left(\sigma_2\left(Y_t^{1,\epsilon,\delta}\right)\right)\right]}{2}
            \biggr)
    \Biggr)\\&=
    \delta B-\frac{1}{2}\left[Y_t^{1,\epsilon,\delta},\delta B\right]\\&\quad
    -\frac{1}{2}\Biggl(
        \left(\epsilon A-\frac{1}{2}\left[\delta B,\epsilon A\right] t\right)^2
        +\frac{\left[\epsilon A-\frac{1}{2}\left[\delta B,\epsilon A\right] t,\epsilon A-\frac{1}{2}\left[\delta B,\epsilon A\right] t\right]}{2}
    \Biggr).
\end{align*}
Counting $\epsilon$ and $\delta$ we can already see that
\begin{align*}
    \left(\epsilon A-\frac{1}{2}\left[\delta B,\epsilon A\right] t\right)^2
    &\approx
    \epsilon^2 A^2,\\
    \left[\epsilon A-\frac{1}{2}\left[\delta B,\epsilon A\right] t,\epsilon A-\frac{1}{2}\left[\delta B,\epsilon A\right] t\right]
    &\approx
    \left[\epsilon A,\epsilon A\right]=0.
\end{align*}
Thus,
\begin{align*}
    \mu\left(Y_t^{1,\epsilon,\delta};2\right)=
    \delta B -\frac{1}{2} \left[\epsilon A, \delta B\right] W_t- \frac{1}{2} \epsilon^2 A^2.
\end{align*}
In total, we have
\begin{align*}
    Y_t^2 =
    \int_{0}^{t}{
        B - \frac{1}{2} \left[A,B\right] W_s  - \frac{1}{2} A^2 ds
    }+
    \int_{0}^{t}{
        A - \frac{1}{2} \left[B,A\right] s dW_s
    }.
\end{align*}
In a next step, we can apply Itô's formula to replace the stochastic integral by a Lebesgue integral like follows:
\begin{lemma}\label{lem:Ito1}%
    For $p,p_1,p_2,q,q_1,q_2 \in \N_0$ we have
    \begin{align}
        &\int_{0}^{t}{s^p W_s^q dW_s}=
            \frac{1}{q+1}\left(
                t^p W_t^{q+1}-
                \int_{0}^{t}{
                    \left[
                        \frac{q(q+1)}{2}s^p W_s^{q-1}+pW_s^{q+1}s^{p-1}
                    \right]
                    ds
                }
            \right)\label{eq:Ito1a}\\
        &\int_{0}^{t}{s^{p_1}\int_{0}^{s}{r^{p_2} W_r^q dr}ds}=
            \frac{1}{1+p_1}\left(
                t^{1+p_1}\int_{0}^{t}{s^{p_2} W_s^q ds}
                - \int_{0}^{t}{s^{1+p_1+p_2}W_s^q ds}
            \right)\label{eq:Ito1b}\\
        &\int_{0}^{t}{s^{p_1}W_s^{q_1}\int_{0}^{s}{r^{p_2} W_r^{q_2} dr}dW_s}=
            \begin{aligned}[t]\arraycolsep=0pt
                \frac{1}{q_1 + 1}\Biggl(&
                t^{p_1} W^{q_1+1}_t \int_{0}^{t}{s^{p_2} W_s^{q_2} ds}
                \\&-\int_{0}^{t}{s^{p_1+p_2} W_s^{q_1+q_2+1} ds}
                \\&- \frac{q_1(q_1+1)}{2}\int_{0}^{t}{s^{p_1}W_s^{q_1-1} \int_{0}^{s}{r^{p_2}W_r^{q_2}dr}ds}
                \\&- p_1 \int_{0}^{t}{s^{p_1-1} W^{q_1+1}_s \int_{0}^{s}{r^{p_2} W_r^{q_2} dr} ds}
                \Biggr).
            \end{aligned}
            \label{eq:Ito1c}.
    \end{align}
\end{lemma}
\begin{proof}
    Note that \eqref{eq:Ito1a} is a special case of \eqref{eq:Ito1c} by setting
    $p=p_1-1$, $q=q_1$ and $p_2=q_2=0$.

    Now, we show \eqref{eq:Ito1b}. With Itô's product rule we get
    \begin{align*}
        d \left(s^{p_1}\int_{0}^{s}{r^{p_2} W_r^q dr}\ s\right)&=
        s^{p_1}\int_{0}^{s}{r^{p_2} W_r^q dr} ds
            + s\ d\left(s^{p_1}\int_{0}^{s}{r^{p_2} W_r^q dr}\right) + 0
        \\&=
        s^{p_1}\int_{0}^{s}{r^{p_2} W_r^q dr} ds
            + s\ d\left(s^{p_1}s^{p_2} W_s^q ds+\int_{0}^{s}{r^{p_2} W_r^q dr} p_1 s^{p_1-1}ds\right)
        \\&=
        \left(1+p_1\right)s^{p_1}\int_{0}^{s}{r^{p_2} W_r^q dr} ds +
        s^{1+p_1+p_2} W_s^q ds.
    \end{align*}
    Rearranging the equation yields the claim.

    Next, we show \eqref{eq:Ito1c}. With Itô's product rule we get
    \begin{align*}
        d\left(
            s^{p_1}W_s^{q_1} \int_{0}^{s}{r^{p_2}W_r^{q_2} dr} W_s
        \right)&=
        s^{p_1}W_s^{q_1} \int_{0}^{s}{r^{p_2}W_r^{q_2} dr} dW_s+
         W_s d\left(s^{p_1}W_s^{q_1} \int_{0}^{s}{r^{p_2}W_r^{q_2} dr}\right)
        \\&\quad+ d\left\langle \cdot ^{p_1}W_\cdot ^{q_1} \int_{0}^{\cdot }{r^{p_2}W_r^{q_2} dr},W_\cdot \right\rangle_s.
    \end{align*}
    Now, use Itô's product rule and Itô's formula on the the following term
    \begin{align*}
        \hspace{1em}&\hspace{-1em}
        d\left(s^{p_1}W_s^{q_1} \int_{0}^{s}{r^{p_2}W_r^{q_2} dr}\right)=
        s^{p_1}W_s^{q_1} d\left(\int_{0}^{s}{r^{p_2}W_r^{q_2} dr}\right)+
        \int_{0}^{s}{r^{p_2}W_r^{q_2} dr} d\left(s^{p_1}W_s^{q_1}\right)
        +0\\&=
        s^{p_1+p_2}W_s^{q_1+q_2} ds+
        \int_{0}^{s}{r^{p_2}W_r^{q_2} dr}
        \Biggl(
            s^{p_1}\biggl(
                q_1 W_s^{q_1-1} dW_s+
                \frac{q_1(q_1-1)}{2} W_s^{q_1-2} ds
            \biggr)\\&\qquad+
            W_s^{q_1}\biggl(
                p_1 s^{p_1-1} ds
            \biggr)+
            0
        \Biggr)\\&=
        \int_{0}^{s}{r^{p_2}W_r^{q_2} dr}s^{p_1}q_1 W_s^{q_1-1} dW_s+
        \Biggl[
            s^{p_1+p_2}W_s^{q_1+q_2} +
            \int_{0}^{s}{r^{p_2}W_r^{q_2} dr}s^{p_1}\frac{q_1(q_1-1)}{2} W_s^{q_1-2}
            \\&\qquad+
            \int_{0}^{s}{r^{p_2}W_r^{q_2} dr}W_s^{q_1}p_1 s^{p_1-1}
        \Biggr] ds
    \end{align*}
    For the quadratic variation from above we have
    \begin{align*}
        d\left\langle \cdot ^{p_1}W_\cdot ^{q_1} \int_{0}^{\cdot }{r^{p_2}W_r^{q_2} dr},W_\cdot \right\rangle_s&=
        \int_{0}^{s}{r^{p_2}W_r^{q_2} dr}s^{p_1}q_1 W_s^{q_1-1} ds.
    \end{align*}
    In total, we have
    \begin{align*}
        \hspace{1em}&\hspace{-1em}
        d\left(
            s^{p_1}W_s^{q_1} \int_{0}^{s}{r^{p_2}W_r^{q_2} dr} W_s
        \right)\\&=
        (1+q_1)s^{p_1}W_s^{q_1} \int_{0}^{s}{r^{p_2}W_r^{q_2} dr} dW_s+
        \Biggl[
            s^{p_1+p_2}W_s^{q_1+q_2+1} +
            \frac{q_1(q_1+1)}{2}s^{p_1} W_s^{q_1-1} \int_{0}^{s}{r^{p_2}W_r^{q_2} dr}
            \\&\qquad+
            p_1\int_{0}^{s}{r^{p_2}W_r^{q_2} dr}W_s^{q_1+1} s^{p_1-1}
        \Biggr] ds.
    \end{align*}
    Rearranging the equation yields the claim.
\end{proof}
Using Lemma \ref{lem:Ito1} \eqref{eq:Ito1a} in the case $p=1$ and $q=0$ and the skew-symmetry of the commutator we have finally
\begin{align*}
    Y_t^2 &=
    Bt - \frac{1}{2} A^2 t + \frac{1}{2} \left[B,A\right] \int_{0}^{t}{W_s ds}
    +AW_t -\frac{1}{2} \left[B,A\right] \left(
        tW_t - \int_{0}^{t}{W_s ds}
    \right)\\&=
    Y_t^1
    - \frac{1}{2} A^2 t + \left[B,A\right] \int_{0}^{t}{W_s ds} -\frac{1}{2} \left[B,A\right] tW_t.
\end{align*}
\paragraph*{Order 3.} From now on we will always repeat the steps seen from the derivation of order 2. First, identify all terms using \eqref{eq:sigmaMagnusCut} and \eqref{eq:muMagnusCut} and secondly apply Lemma \ref{lem:Ito1} to remove the iterated (stochastic) integrals by expressions with single Lebesgue integrals.

After using \eqref{eq:sigmaMagnusCut} and \eqref{eq:muMagnusCut} we get
\begin{align*}
    \sigma_3\left(Y_t^{2,\epsilon,\delta}\right)&=
    \epsilon A - \frac{1}{2}\left[\delta B,\epsilon A\right]t - \frac{1}{12} \left[\left[\delta B,\epsilon A\right],B\right] t^2
    -\frac{1}{2}\left[\left[\delta B,\epsilon A\right],\epsilon A\right] \int_{0}^{t}{W_s ds}
    \\&\quad+\frac{1}{6}\left[\left[\delta B,\epsilon A\right],\epsilon A\right] W_t
    \\
    \mu_3\left(Y_t^{2,\epsilon,\delta}\right)&=
    \delta B - \frac{1}{2}\epsilon^2 A^2
    +\frac{1}{12}\left[\left[\delta B,\epsilon A\right],\epsilon A\right]t
    -\frac{1}{2}\left[\left[\delta B,\epsilon A\right],\delta B\right]\int_{0}^{t}{W_s ds}
    +\frac{1}{2}\left[\delta B,\epsilon A\right]W_t
    \\&\quad+\frac{1}{3}\left[\left[\delta B,\epsilon A\right],\delta B\right] t W_t
    +\frac{1}{12}\left[\left[\delta B,\epsilon A\right],\epsilon A\right] W_t^2
\end{align*}
Applying Lemma \ref{lem:Ito1} and collecting all terms yields
\begin{align*}
    Y_t^3 =
    Y_t^2&+
    \left[\left[B,A\right],A\right] \left(\frac{1}{2}\int_{0}^{t}{W_s^2 ds}-\frac{1}{2}W_t \int_{0}^{t}{W_s ds}+\frac{1}{12} t W_t^2\right)
    \\&+\left[\left[B,A\right],B\right] \left(
        \int_{0}^{t}{sW_s ds}-\frac{1}{2} t \int_{0}^{t}{W_s ds}-\frac{1}{12} t^2 W_t
    \right).
\end{align*}
\begin{footnotesize}
\bibliographystyle{acm}
\bibliography{bib}
\end{footnotesize}

\end{document}